\documentclass[reqno]{amsart}
\usepackage{amssymb,amsmath,amsthm}
\newtheorem{theorem}{Theorem}[section]
\newtheorem{lemma}[theorem]{Lemma}
\newtheorem{corollary}[theorem]{Corollary}
\newtheorem{proposition}[theorem]{Proposition}

\numberwithin{equation}{section}

\newcommand{\spt}[1]{\mbox{\normalfont spt}\Parans{#1}}
\newcommand{\sptBar}[2]{\overline{\mbox{\normalfont spt}}_{#1}\Parans{#2}}

\newcommand{\Parans}[1]{\left(#1\right)}

\newcommand{\aqprod}[3]{\Parans{#1;#2}_{#3}}
\newcommand{\jacprod}[2]{\left[#1;#2\right]_\infty}

\newcommand{\bin}[2]{
	\Parans{\substack{#1\\#2}}
}

\newcommand{\ST}{\widetilde{ \mbox{\normalfont PP}}}
\newcommand{\sT}{\widetilde{ \mbox{\normalfont pp}}}

\newcommand{\J}[2]{j\Parans{#1;#2}}
\newcommand{\SSeries}[3]{\Sigma\Parans{#1,#2,#3}}

\author{CHRIS JENNINGS-SHAFFER}
\address{Department of Mathematics, University of Florida\\
Gainesville, Florida 32611, USA
\endgraf cjenningsshaffer@ufl.edu}



\title{Congruences for Partition Pairs with Conditions}

\begin{document}
\allowdisplaybreaks

\begin{abstract}
We prove congruences for the number of partition pairs
$(\pi_1,\pi_2)$ such that $\pi_1$ is non-empty,
$s(\pi_1)\le s(\pi_2)$, and $\ell(\pi_2)< 2s(\pi_1)$ where
$s(\pi)$ is the smallest part and $\ell(\pi)$ is the largest part of
a partition.
The proofs use Bailey's Lemma and a generalized Lambert series identity of Chan.
We also discuss how a partition pair crank gives combinatorial refinements
of these congruences.
\end{abstract}

\maketitle

\section{Introduction}

We recall that a partition of a non-negative integer $n$ is a non-decreasing 
sequence of positive integers that sum to $n$. For a partition $\pi$ we 
denote the sum of parts by $|\pi|$, the number of parts by $\#(\pi)$,
the smallest part by $s(\pi)$, and the largest part by $\ell(\pi)$. We denote
the empty partition by $\emptyset$ and use the convention that the empty
partition has smallest part $\infty$ and largest part 0.
A partition pair of $n$ is an ordered pair of partitions
$(\pi_1,\pi_2)$ such that $n=|\pi_1|+|\pi_2|$.

In this paper, we consider the set of
partition pairs
$(\pi_1,\pi_2)$ such that $\pi_1\not =\emptyset$, $s(\pi_1)\le s(\pi_2)$, and
$\ell(\pi_2)<2s(\pi_1)$. That is, $\pi_1$ is non-empty and each part of $\pi_2$
is in the interval $[s(\pi_1),2s(\pi_1))$. We denote the set of all such partition
pairs by $\ST$. As we will see from the generating
function, the number of such partition pairs of $n$ can
also be viewed as the number
of occurrences of the smallest parts in the partition pairs $(\pi_1,\pi_2)$ of 
$n$ such that $\pi_1\not =\emptyset$, $s(\pi_1)< s(\pi_2)$, and
$\ell(\pi_2)<2s(\pi_1)$.

We let $\sT(n)$ denote the number of partition pairs $(\pi_1,\pi_2)$
from $\ST$ 
of $n$. As an example,  $\sT(5) = 15$ 
since the partition pairs are:
$(5, 			\emptyset)$,
$(1+4, 		\emptyset)$,
$(2+3, 		\emptyset)$,
$(1+1+3, 		\emptyset)$,
$(1+2+2, 		\emptyset)$,
$(1+1+1+2, 	\emptyset)$,
$(1+1+1+1+1, \emptyset)$,
$(1+3, 		1)$,
$(1+1+2, 		1)$,
$(1+1+1+1, 	1$),
$(1+2, 		1+1)$,
$(1+1+1, 		1+1)$, 
$(2, 			3)$,
$(1+1, 		1+1+1)$, and
$(1, 		1+1+1+1)$.

These partition pairs actually came from considering smallest parts partition
functions.
We recall that $\spt{n}$ is the number of occurrences of the
smallest part in each partition of $n$. 
An overpartion of $n$ is a partition of $n$ in which the first 
occurrence of a part may be overlined; the number of smallest parts among the
overpartitions of $n$ where the smallest part is not overlined is
$\sptBar{}{n}$. We use the convention of not including overpartitions when the
smallest part is overlined, 
otherwise if we did include these overpartitions, then
we would exactly double $\sptBar{}{n}$.
Additionally we have $\sptBar{1}{n}$ and $\sptBar{2}{n}$
denoting the restriction of $\sptBar{}{n}$ to overpartitions where the smallest part 
is odd and even, respectively. 

In Section 3 of \cite{GarvanJennings}, Garvan and the author considered the partition pairs $(\pi_1,\pi_2)$
such that 
$\pi_1$ is non-empty, $s(\pi_1)\le s(\pi_2)$, and
if a part of $\pi_2$ is larger than $2s(\pi_1)$ then it must be odd.
It turns out the number of such partition pairs of $n$ is
$\sptBar{}{n}$. Additionally restricting $s(\pi_1)$ to an odd
or even number gives $\sptBar{1}{n}$ or $\sptBar{2}{n}$, respectively.
The partition pairs here are those were we additionally do not allow the parts 
of $\pi_2$ to be larger than $2s(\pi_1)-1$.

The smallest parts functions satisfy a wide array of congruences. For example, 
for $n\ge 0$ we have
\begin{align*}
	\spt{5n+4} &\equiv 0 \pmod{5},
	\\
	\spt{7n+5} &\equiv 0 \pmod{7},
	\\
	\spt{13n+6} &\equiv 0 \pmod{13},
	\\
	\sptBar{}{3n} & \equiv 0 \pmod{3},
	\\
	\sptBar{1}{3n} &\equiv 0 \pmod{3},
	\\
	\sptBar{1}{5n} &\equiv 0 \pmod{5},
	\\
	\sptBar{2}{3n} &\equiv 0 \pmod{3},
	\\
	\sptBar{2}{3n+1} &\equiv 0 \pmod{3},
	\\
	\sptBar{2}{5n+3} &\equiv 0 \pmod{5}
.
\end{align*}
The congruences for $\spt{n}$ were first proved by Andrews in \cite{Andrews}
and
the congruences for $\sptBar{}{n}$, $\sptBar{1}{n}$, and $\sptBar{2}{n}$ 
were
first proved by Bringmann, Lovejoy, and Osburn in \cite{BLO}.
The function $\sT(n)$ satisfies similar congruences.

\begin{theorem}\label{TheoremCongruences}
For $n\ge 0$
\begin{align*}
	\sT(3n+2) &\equiv 0 \pmod{3},\\ 	
	\sT(5n+3) &\equiv 0 \pmod{5},\\ 	
	\sT(5n+4) &\equiv 0 \pmod{5}. 	
\end{align*}
\end{theorem}

To start, by summing according to the smallest part of $\pi_1$ we see a generating function
for $\sT(n)$ is
\begin{align*}
	\ST(q) 
	&= 
	\sum_{n=0}^\infty \sT(n)q^n
	=
	\sum_{n=1}^\infty \frac{q^n}{\aqprod{q^{n}}{q}{\infty}
		\aqprod{q^{n}}{q}{n}}
	=
		\sum_{n=1}^\infty \frac{q^n\aqprod{q^{2n}}{q}{\infty}}
			{\aqprod{q^n,q^n}{q}{\infty}}
.
\end{align*}
By rewriting the summand as 
$\frac{q^n}{(1-q^n)^2\aqprod{q^{n+1}}{q}{\infty}\aqprod{q^{n+1}}{q}{n-1}}$,
we see that $\sT(m)$
is also the number of occurrences of the
smallest parts in the partition pairs $(\pi_1,\pi_2)$ of $m$ where each part of 
$\pi_2$ is strictly between $s(\pi_1)$ and $2s(\pi_2)$. 
It is not clear if $\sT(n)$
can easily be phrased in
terms of the smallest parts of any sort of restricted partitions or 
overpartitions, rather than partition pairs.
Here 
and throughout, we use the following standard
$q$-hypergeometric notation
\begin{align*}
	\aqprod{z}{q}{n} &= \prod_{k=0}^{n-1} (1-zq^k),
	\\
	\aqprod{z}{q}{\infty} &= \prod_{k=0}^{\infty} (1-zq^k),
	\\
	\aqprod{z_1,\dots,z_n}{q}{\infty} &= \aqprod{z_1}{q}{\infty}\dots\aqprod{z_n}{q}{\infty},
	\\
	\jacprod{z}{q} &= \aqprod{z,q/z}{q}{\infty},
	\\
	\jacprod{z_1,\dots,z_n}{q} &= \jacprod{z_1}{q}\dots\jacprod{z_n}{q}.
\end{align*}

We consider a two-variable generalization of $\ST(q)$ given by
\begin{align*}
	\ST(z,q)
	&=	
	\sum_{n=1}^\infty \frac{q^n\aqprod{q^{2n}}{q}{\infty}}
			{\aqprod{zq^n,z^{-1}q^n}{q}{\infty}}
	=
	\sum_{n=1}^\infty\sum_{m=-\infty}^\infty C(m,n)z^mq^n	
,
\end{align*}
so that $\ST(1,q) = \ST(q)$. We define
\begin{align*}
	C(k,t,n) &= \sum_{m\equiv k\pmod{t}} C(m,n)
,
\end{align*}
and so for any positive $t$
\begin{align}\label{EqPPinTermsOfCranks}
	\sT(n) &= \sum_{m=-\infty}^\infty C(m,n)
	= \sum_{k=0}^{t-1} C(k,t,n)
.
\end{align}

We prove Theorem \ref{TheoremCongruences} with the following strategy.
If $\zeta_t$ is a primitive $t^{th}$ root of unity
with $t=3$ or $5$, then
\begin{align*}
	\ST(\zeta_t,q)
	&=
	\sum_{n=1}^\infty\sum_{m=-\infty}^\infty C(m,n)\zeta_t^mq^n	
	=
	\sum_{n=1}^\infty q^n \sum_{k=0}^{t-1} \zeta_t^k C(k,t,n)
.		
\end{align*}
As the minimal polynomial for $\zeta_t$ is $\sum_{k=0}^{t-1}x^k$,
if the coefficient of $q^N$ in $\ST(\zeta_t)$ is zero,
then
$C(0,t,N) = C(1,t,N) \dots = C(t-1,t,N)$.
Thus, by (\ref{EqPPinTermsOfCranks}),
$\sT(N) = t\cdot C(0,t,N)$ and so
$\sT(N)\equiv 0\pmod{t}$.

Theorem \ref{TheoremCongruences} now follows from showing that the 
coefficient of $q^{3n+2}$
in $\ST(\zeta_3,q)$ and the coefficient of $q^{5n+3}$ and $q^{5n+4}$
in $\ST(\zeta_5,q)$ are zero.
In stating our theorems, we need to first define the series:
\begin{align*}
	\Sigma(z,w,q) &= \sum_{n=-\infty}^\infty \frac{q^{2n(n+1)}w^n}{1-zq^n}. 
\end{align*}
For integers $a,b,c$, we will write
\begin{align*}
	\Sigma(a,b,c) &= \Sigma(q^a,q^b,q^c)
.
\end{align*}
We will prove the following theorems about $\ST(\zeta_3,q)$ and 
$\ST(\zeta_5,q)$.
\begin{theorem}\label{theorem2}
\begin{align*}
	\ST(\zeta_3,q)
	&=
	A_0(q^3) + qA_1(q^3) + q^2A_2(q^3)
,
\end{align*}
where
\begin{align*}
	A_0(q) 
	&= 
		\frac{\aqprod{q^{9}}{q^{9}}{\infty}\jacprod{q^3}{q^{9}}}
			{\jacprod{q}{q^{9}}^2\jacprod{q^{4}}{q^{9}}}
		-
		\frac{1}{\aqprod{q^{9}}{q^{9}}{\infty}\jacprod{q^2}{q^{9}}}
		\Parans{\Sigma(1,-14,9) +q\Sigma(1,-5,9)}
		\\&\quad
		-
		\frac{1}{\aqprod{q^{9}}{q^{9}}{\infty}\jacprod{q^{4}}{q^{9}}}
		\Parans{q\Sigma(1,-8,9) +q^5\Sigma(4,1,9)}
	,\\
	A_1(q) &= \frac{\aqprod{q^3}{q^3}{\infty}}{\jacprod{q}{q^3}}
	,\\
	A_2(q) &= 0.
\end{align*}
\end{theorem}

\begin{theorem}\label{theorem3}
\begin{align*}
	\ST(\zeta_5,q)
	&=
	B_0(q^5) + qB_1(q^5) + q^2B_2(q^5) + q^3B_3(q^5) + q^4B_4(q^5)
,
\end{align*}
where
\begin{align*}
	B_0(q) 
	&= 
		-\frac{\zeta_5+\zeta_5^4}{\aqprod{q^5}{q^5}{\infty}}
		\Parans{q^{15}\Sigma(7,10,15) +q^{25}\Sigma(13,25,15)}
	,\\
	B_1(q) 
	&= 
		-\frac{1}{\aqprod{q^5}{q^5}{\infty}}
		\Parans{q^{13}\Sigma(10,10,15) +q^{23}\Sigma(10,25,15)}
		+\frac{\aqprod{q^3}{q^3}{\infty}^3}
			{\aqprod{q}{q}{\infty}\aqprod{q^5}{q^5}{\infty}}
	,\\
	B_2(q) 
	&= 
		\frac{1+\zeta_5+\zeta_5^4}{\aqprod{q^5}{q^5}{\infty}}
		\Parans{\Sigma(1,-20,15) +q^4\Sigma(4,-5,15)}	
	,\\
	B_3(q) &= 0
	,\\
	B_4(q) &= 0
.
\end{align*}
\end{theorem}

The partition pairs from Section 3 of \cite{GarvanJennings}, 
were an intermediate step to defining a crank on the smallest parts
of overpartitions. This crank gave 
a combinatorial interpretation of the congruences for
$\sptBar{}{n}$, $\sptBar{1}{n}$, and $\sptBar{2}{n}$.
In \cite{GarvanJennings} we let $k(\pi_1,\pi_2)$ denote the number of parts 
of $\pi_2$ that are between $s(\pi_1)$ and $2s(\pi_1)-1$
and defined
\begin{align*}
	\overline{\mbox{crank}}(\pi_1,\pi_2)
	&=
	\left\{
	\begin{array}{ll}
		(\#\mbox{ of parts of }\pi_1) - 1 		
		& \mbox{ if } k(\pi_1,\pi_2) = 0
		\\
		(\#\mbox{ of parts of }\pi_1 \ge s(\pi_1)+k(\pi_1,\pi_2)) - k(\pi_1,\pi_2) 		
		& \mbox{ if } k(\pi_1,\pi_2) > 0.
	\end{array}
	\right.
\end{align*}
For the partition pairs in this article, $k(\pi_1,\pi_2)$ is $\#(\pi_2)$
and so we define
\begin{align*}
	\mbox{paircrank}(\pi_1,\pi_2)
	&=
	\left\{
	\begin{array}{ll}
		(\#\mbox{ of parts of }\pi_1) - 1 		
		& \mbox{ if } \#(\pi_2) = 0
		\\
		(\#\mbox{ of parts of }\pi_1 \ge s(\pi_1)+\#(\pi_2)) - \#(\pi_2) 		
		& \mbox{ if } \#(\pi_2) > 0.
	\end{array}
	\right.
\end{align*}
This paircrank yields a combinatorial refinement of the congruences for
$\sT(n)$.

\begin{theorem}\label{crankTheorem}
\begin{enumerate}
\item[(i)]
The residue of the ${\normalfont\mbox{paircrank}}$ mod $3$ divides the partition pairs
from $\ST$
of $3n+2$ into $3$ equal classes.
\item[(ii)]
The residue of the ${\normalfont\mbox{paircrank}}$ mod $5$ divides the partition pairs
from $\ST$ of $5n+3$ into $5$ equal classes.
\item[(iii)]
The residue of the ${\normalfont\mbox{paircrank}}$ mod $5$ divides the partition pairs
from $\ST$ of $5n+4$ into $5$ equal classes.
\end{enumerate}
\end{theorem}

In Section 2, we provide an alternative expression for $\ST(z,q)$. 
In Sections 3 and 4, we prove Theorems \ref{theorem2} and
\ref{theorem3}, respectively. In Section 5, we prove
Theorem \ref{crankTheorem}
by showing that $C(m,n)$ is the number of partition
pairs from $\ST$ of $n$ with paircrank $m$.

\section{Preliminaries}

Before proving Theorems $\ref{theorem2}$ and $\ref{theorem3}$, we require
another form for $\ST(z,q)$. We recall a pair of sequences
$(\alpha,\beta)$ are a Bailey pair relative to $(a,q)$ if
\begin{align*}
	\beta_n &= \sum_{k=0}^n \frac{\alpha_k}{\aqprod{q}{q}{n-k}\aqprod{aq}{q}{n+k}}
.
\end{align*}
A limiting case of Bailey's lemma states that if $(\alpha,\beta)$ are a 
Bailey pair for $(a,q)$, then
\begin{align*}
	\sum_{n=0}^\infty 
		\aqprod{\rho_1,\rho_2}{q}{n}\Parans{\frac{aq}{\rho_1\rho_2}}^n\beta_n
	&=
	\frac{\aqprod{aq/\rho_1,aq/\rho_2}{q}{\infty}}{\aqprod{aq,aq/(\rho_1\rho_2)}{q}{\infty}}
	\sum_{n=0}^\infty 
	\frac{\aqprod{\rho_1,\rho_2}{q}{n}\Parans{\frac{aq}{\rho_1\rho_2}}^n\alpha_n}
		{\aqprod{aq/\rho_1,aq/\rho_2}{q}{n}}
.
\end{align*}

\begin{proposition}\label{PropositionBaileyPair}
The pair $(\alpha,\beta)$ where
\begin{align*}
	\beta_n &= \frac{1}{\aqprod{q}{q}{2n-1}},
	&
	\alpha_{3n} &= 0, 
	&
	\alpha_{3n\pm 1} &=	q^{6n^2\pm n} - q^{6n^2\pm 7n + 2},
\end{align*}
is a Bailey pair relative to $(1,q)$. Here 
$\alpha_0=\beta_0=0$.
\end{proposition}
\begin{proof}
The following is (2.1) of \cite{Slater},
\begin{align*}
	&\sum_{n=-\infty}^\infty
	\frac{(1-aq^{2n})\aqprod{b,c,d,e}{q}{n}a^{2n}q^n}
	{(1-a)\aqprod{aq/b,aq/c,aq/d,aq/e}{q}{n}b^nc^nd^ne^n}
	\\
	&=
	\frac{\aqprod{q,q/a,qa,aq/bc,aq/,bd,aq/be,aq/cd,aq/de}{q}{\infty}}
	{\aqprod{q/b,q/c,q/d,q/e,aq/b,aq/c,aq/d,aq/e,a^2q/bcde}{q}{\infty}}
.
\end{align*}
With $q\mapsto q^3$, $b=q^{1-N}$, $c=q^{2-N}$, $d=q^{3-N}$, $a=q^2$, and
$e\rightarrow\infty$, this becomes
\begin{align*}
	\sum_{n=-\infty}^\infty
	\frac{(1-q^{6n+2})\aqprod{q^{1-N}}{q}{3n}(-1)^nq^{\frac{3n(n-1)}{2} + n + 3nN}}
		{(1-q^2)\aqprod{q^{2+N}}{q}{3n}}
	&=
	\frac{\aqprod{q^3,q,q^5}{q^3}{\infty}\aqprod{q^{2N}}{q}{\infty}}
	{\aqprod{q^N,q^{N+2}}{q}{\infty}}
.
\end{align*}
Here the sum is actually finite as we see the summands are eventually zero for large
positive and large negative values of $n$.
With this in mind, we have for $N>0$
\begin{align*}
	&\sum_{n=0}^N \frac{\alpha_n}{\aqprod{q}{q}{N+n}\aqprod{q}{q}{N-n}}
	\\
	&=
	\frac{\alpha_1}{\aqprod{q}{q}{N-1}\aqprod{q}{q}{N+1}}
	+
	\sum_{n=1}^\infty
		\frac{\alpha_{3n-1}}{\aqprod{q}{q}{N-3n+1}\aqprod{q}{q}{N+3n-1}}
		+
		\frac{\alpha_{3n+1}}{\aqprod{q}{q}{N-3n-1}\aqprod{q}{q}{N+3n+1}}
	\\
	&=
	\sum_{n=-\infty}^\infty \frac{q^{6n^2+n} - q^{6n^2+7n+2}}
		{\aqprod{q}{q}{N-3n-1}\aqprod{q}{q}{N+3n+1}}
	\\
	&=
	\sum_{n=-\infty}^\infty 
		\frac{q^{6n^2+n}(1-q^{6n+2})\aqprod{q^{1-N}}{q}{3n} 
			(-1)^n q^{ -\bin{3n}{2} + 3n(N-1)}
		}
		{\aqprod{q}{q}{N-1}\aqprod{q}{q}{N+1}\aqprod{q^{N+2}}{q}{3n}}
	\\
	&=
	\frac{1}{\aqprod{q}{q}{N-1}\aqprod{q}{q}{N+1}}
	\sum_{n=-\infty}^\infty 
		\frac{q^{6n^2+n}(1-q^{6n+2})\aqprod{q^{1-N}}{q}{3n} 
			(-1)^n q^{ \frac{3n(n-1)}{2} + n + 3nN}
		}
		{\aqprod{q^{N+2}}{q}{3n}}
	\\
	&= 
	\frac{(1-q^2)}{\aqprod{q}{q}{N-1}\aqprod{q}{q}{N+1}}
	\cdot
	\frac{\aqprod{q^3,q,q^5}{q^3}{\infty}\aqprod{q^{2N}}{q}{\infty}}
	{\aqprod{q^N,q^{N+2}}{q}{\infty}}
	\\
	&= 
	\frac{\aqprod{q,q^{2N}}{q}{\infty}}
	{\aqprod{q,q}{q}{\infty}}
	\\
	&=
	\frac{1}{\aqprod{q}{q}{2N-1}}
.
\end{align*}

Thus, the result follows.
\end{proof}

\begin{corollary}
\begin{align*}
	\ST(z,q)
	&=
	\frac{1}{\aqprod{q}{q}{\infty}}	
	\sum_{n=-\infty}^\infty \frac{q^{6n^2+4n+1}(1-q^{6n+2})}
		{(1-zq^{3n+1})(1-z^{-1}q^{3n+1})}
.
\end{align*}
\end{corollary}
\begin{proof}
By Proposition \ref{PropositionBaileyPair},
we have
\begin{align*}
	\ST(z,q)
	&=
	\frac{\aqprod{q}{q}{\infty}}{\aqprod{z,z^{-1}}{q}{\infty}}
	\sum_{n=1}^\infty \frac{\aqprod{z,z^{-1}}{q}{n}q^n}{\aqprod{q}{q}{2n-1}}
	\\
	&=
	\frac{\aqprod{q}{q}{\infty}}{\aqprod{z,z^{-1}}{q}{\infty}}
	\sum_{n=0}^\infty \aqprod{z,z^{-1}}{q}{n}q^n\beta_n
	\\
	&=
	\frac{1}{(1-z)(1-z^{-1})\aqprod{q}{q}{\infty}}
	\sum_{n=0}^\infty \frac{\aqprod{z,z^{-1}}{q}{n}q^n\alpha_n}
		{\aqprod{zq,z^{-1}q}{q}{n}}
	\\
	&=
	\frac{1}{\aqprod{q}{q}{\infty}}
	\sum_{n=0}^\infty \frac{q^n\alpha_n}{(1-zq^n)(1-z^{-1}q^n)}
	\\
	&=
	\frac{1}{\aqprod{q}{q}{\infty}}	
	\sum_{n=0}^\infty \frac{q^{3n+1}\alpha_{3n+1}}
		{(1-zq^{3n+1})(1-z^{-1}q^{3n+1})}
	+
	\sum_{n=1}^\infty \frac{q^{3n-1}\alpha_{3n-1}}
		{(1-zq^{3n-1})(1-z^{-1}q^{3n-1})}
	\\
	&=
	\frac{1}{\aqprod{q}{q}{\infty}}	
	\sum_{n=0}^\infty \frac{q^{6n^2+4n+1}(1-q^{6n+2})}
		{(1-zq^{3n+1})(1-z^{-1}q^{3n+1})}
	+
	\sum_{n=-\infty}^{-1} \frac{q^{6n^2-2n-1}(1-q^{6n+2})}
		{(1-zq^{-3n-1})(1-z^{-1}q^{-3n-1})}
	\\
	&=
	\frac{1}{\aqprod{q}{q}{\infty}}	
	\sum_{n=-\infty}^\infty \frac{q^{6n^2+4n+1}(1-q^{6n+2})}
		{(1-zq^{3n+1})(1-z^{-1}q^{3n+1})}
.
\end{align*}
\end{proof}

Next we define
\begin{align*}
	U_\ell(b) 
	&= 
	\sum_{n=-\infty}^\infty \frac{q^{6n^2+bn}}{1-q^{\ell(3n+1)}}
.
\end{align*}

We use the additional product notation
\begin{align*}
	j(z;q) &= \jacprod{z}{q}\aqprod{q}{q}{\infty}
	= \sum_{n=-\infty}^\infty (-1)^n z^n q^{n(n-1)/2}.
\end{align*}

We now provide properties of $\jacprod{z}{q}$,  $j(z,q)$, 
and $\SSeries{z}{w}{q}$. The proofs are routine and thus omitted.
\begin{proposition}\label{PropositionMiscIdents}
\begin{align}
	\label{JacProdProperty1}
	\jacprod{z}{q} &= \jacprod{q/z}{q}
	,\\
	\label{JacProdProperty2}
	\jacprod{z}{q} &= -z\jacprod{qz}{q}
	,\\
	\label{JacProdProperty3}
	\jacprod{z}{q} &= -z\jacprod{z^{-1}}{q}
	,\\
	\label{SigmaProperty1}
	\SSeries{z}{w}{q} &= -z^{-1}\SSeries{z^{-1}}{w^{-1}q^{-3}}{q}
	,\\
	\label{SigmaProperty3}
	\SSeries{z}{w}{q}&= -z^{-1}w^{-1}q\SSeries{z^{-1}q}{w^{-1}q}{q}
	,\\
	\label{SigmaProperty4}
	\SSeries{z}{z^4/q^2}{q} + z\SSeries{z}{z^4/q}{q}
		&=
		z^{-1}\SSeries{z}{z^4/q^3}{q} + z^2\SSeries{z}{z^4}{q}
		-z^{-1}j(q/z^2;q)
	.
\end{align} 
\end{proposition}

\section{Proof of Theorem \ref{theorem2}}

\begin{lemma}\label{ChanLemma1}
\begin{align*}
	&\frac{\aqprod{q}{q}{\infty}^2}{\jacprod{b_1,1/b_1,b_3,b_4}{q}}
	\\
	&=
		\frac{1}{\jacprod{1/b_1^2,b_3/b_1,b_4/b_1}{q}}
			\SSeries{b_1}{\frac{b_1^4}{b_3b_4}}{q}
		-
		\frac{b_1}{\jacprod{b_1^2,b_3b_1,b_4b_1}{q}}
			\SSeries{b_1}{\frac{b_1^4b_3b_4}{q^3}}{q}
		\\&\quad
		+
		\frac{1}{\jacprod{b_1/b_3,1/b_1b_3,b_4/b_3}{q}}
			\SSeries{b_3}{\frac{b_3^3}{b_4}}{q}
		+
		\frac{1}{\jacprod{b_1/b_4,1/b_1b_4,b_3/b_4}{q}}
			\SSeries{b_4}{\frac{b_4^3}{b_3}}{q}
.
\end{align*}
\end{lemma}
\begin{proof}
We will use the $s=4$, $r=0$ case of Theorem 2.1 in \cite{Chan}:
\begin{align*}
	&\frac{\aqprod{q}{q}{\infty}^2}{\jacprod{b_1,b_2,b_3,b_4}{q}}
	\\
	&=
		\frac{1}{\jacprod{b_2/b_1,b_3/b_1,b_4/b_1}{q}}
			\SSeries{b_1}{\frac{b_1^3}{b_2b_3b_4}}{q}
		+
		\frac{1}{\jacprod{b_1/b_2,b_3/b_2,b_4/b_2}{q}}
			\SSeries{b_2}{\frac{b_2^3}{b_1b_3b_4}}{q}
		\\&\quad
		+
		\frac{1}{\jacprod{b_1/b_3,b_2/b_3,b_4/b_3}{q}}
			\SSeries{b_3}{\frac{b_3^3}{b_1b_2b_4}}{q}
		+
		\frac{1}{\jacprod{b_1/b_4,b_2/b_4,b_3/b_4}{q}}
			\SSeries{b_4}{\frac{b_4^3}{b_1b_2b_3}}{q}
.
\end{align*}
Setting $b_2 = b_1^{-1}$ we get
\begin{align*}
	&\frac{\aqprod{q}{q}{\infty}^2}{\jacprod{b_1,1/b_1,b_3,b_4}{q}}
	\\
	&=
		\frac{1}{\jacprod{1/b_1^2,b_3/b_1,b_4/b_1}{q}}
			\SSeries{b_1}{\frac{b_1^4}{b_3b_4}}{q}
		+
		\frac{1}{\jacprod{b_1^2,b_3b_1,b_4b_1}{q}}
			\SSeries{b_1^{-1}}{\frac{1}{b_1^4b_3b_4}}{q}
		\\&\quad
		+
		\frac{1}{\jacprod{b_1/b_3,1/b_1b_3,b_4/b_3}{q}}
			\SSeries{b_3}{\frac{b_3^3}{b_4}}{q}
		+
		\frac{1}{\jacprod{b_1/b_4,1/b_1b_4,b_3/b_4}{q}}
			\SSeries{b_4}{\frac{b_4^3}{b_3}}{q}
.
\end{align*}
By (\ref{SigmaProperty1}) 
\begin{align*}
	\SSeries{b_1^{-1}}{\frac{1}{b_1^4b_3b_4}}{q}
	&=
	-b_1\SSeries{b_1}{\frac{b_1^4b_3b_4}{q^3}}{q}
,
\end{align*}
and so the result follows.
\end{proof}

Additionally setting $b_4 = b_3^{-1}$ gives the following Lemma.
\begin{lemma}\label{ChanLemma2}
\begin{align*}
	&\frac{\aqprod{q}{q}{\infty}^2}{\jacprod{b_1,1/b_1,b_3,1/b_3}{q}}
	\\
	&=
		\frac{1}{\jacprod{1/b_1^2,b_3/b_1,1/b_1b_3}{q}}
			\SSeries{b_1}{b_1^4}{q}
		-
		\frac{b_1}{\jacprod{b_1^2,b_3b_1,b_1/b_3}{q}}
			\SSeries{b_1}{\frac{b_1^4}{q^3}}{q}
		\\&\quad
		+
		\frac{1}{\jacprod{b_1/b_3,1/b_1b_3,1/b_3^2}{q}}
			\SSeries{b_3}{b_3^4}{q}
		-
		\frac{b_3}{\jacprod{b_1b_3,b_3/b_1,b_3^2}{q}}
			\SSeries{b_3}{\frac{b_3^4}{q^3}}{q}
.
\end{align*}
\end{lemma}

We now require the following Propositions.

\begin{proposition}\label{Prop1Zeta3}
\begin{align}
	\label{3DissectionEq1}
	&\Sigma(1,-11,9)
	+q^{15}\Sigma(7,16,9)
	\nonumber\\
	&=
	\frac{\aqprod{q^9}{q^{9}}{\infty}^2\jacprod{q^3}{q^9}}
	{\jacprod{q}{q^{9}}\jacprod{q^4}{q^{9}}}
	-
	\frac{\jacprod{q}{q^9}}{\jacprod{q^4}{q^9}}
	\Parans{q\Sigma(1,-8,9) + q^{5}\Sigma(4,1,9)}
	,\\
	\label{3DissectionEq2}
	&q^{6}\Sigma(4,4,9)
	+q^{13}\Sigma(7,13,9)
	\nonumber\\
	&=
	q\frac{\aqprod{q^9}{q^{9}}{\infty}^2\jacprod{q^3}{q^9}}
	{\jacprod{q^4}{q^{9}}^2}
	-
	\frac{\jacprod{q^2}{q^9}}{\jacprod{q^4}{q^9}}
	\Parans{q\Sigma(1,-8,9)+q^5\Sigma(4,1,9)}
	.
\end{align}
\end{proposition}
\begin{proof}
For (\ref{3DissectionEq1}) we move all $\SSeries{z}{w}{q}$ terms to one side,
divide by $\jacprod{q,q^2,q^3}{q^9}$, and use that
$\SSeries{4}{1}{9} = -q^{4}\SSeries{5}{8}{9}$
by (\ref{SigmaProperty3}). Equation (\ref{3DissectionEq1}) is then equivalent
to
\begin{align}\label{3DissectionEq1a}
	\frac{\aqprod{q^9}{q^9}{\infty}^2}{\jacprod{q,q,q^2,q^4}{q^9}}
	&=
	\frac{1}{\jacprod{q,q^2,q^3}{q^9}}\Sigma(1,-11,9) 
	+\frac{q^{15}}{\jacprod{q,q^2,q^3}{q^9}}\Sigma(7,16,9)
	\nonumber\\&\quad
	+\frac{q}{\jacprod{q^2,q^3,q^4}{q^9}}\Sigma(1,-8,9) 
	-\frac{q^9}{\jacprod{q^2,q^3,q^4}{q^9}}\Sigma(5,8,9)
.
\end{align}
In Lemma \ref{ChanLemma1} we use $q\mapsto q^9$, $b_1=q, b_3 = q^7, b_4=q^5$ to get
\begin{align*}
	\frac{\aqprod{q^9}{q^9}{\infty}^2}{\jacprod{q,q^{-1},q^7,q^5}{q^9}}
	=&
		\frac{1}{\jacprod{q^{-2},q^{6},q^4}{q^9}}\SSeries{1}{-8}{9}
		-
		\frac{q}{\jacprod{q^{2},q^{8},q^{6}}{q^9}}\SSeries{1}{-11}{9}
		\\&
		+
		\frac{1}{\jacprod{q^{-6},q^{-8},q^{-2}}{q^9}}\SSeries{7}{16}{9}
		+
		\frac{1}{\jacprod{q^{-4},q^{-6},q^{2}}{q^9}}\SSeries{5}{8}{9}
.
\end{align*}
Simplifying the products with Proposition \ref{PropositionMiscIdents} yields
(\ref{3DissectionEq1a}).

For (\ref{3DissectionEq2}) we move all $\SSeries{z}{w}{q}$ terms to one side,
divide by $q\jacprod{q,q^2,q^3}{q^9}$, and use that
$\SSeries{1}{-8}{9} = -q^{16}\SSeries{8}{17}{9}$
by (\ref{SigmaProperty3}). Equation (\ref{3DissectionEq2}) is then equivalent
to
\begin{align}\label{3DissectionEq2a}
	\frac{\aqprod{q^9}{q^9}{\infty}^2}{\jacprod{q,q^2,q^4,q^4}{q^9}}
	&=
	-\frac{q^{16}}{\jacprod{q,q^3,q^4}{q^9}}\Sigma(8,17,9) 
	+\frac{q^{4}}{\jacprod{q,q^3,q^4}{q^9}}\Sigma(4,1,9)
	\nonumber\\&\quad
	+\frac{q^5}{\jacprod{q,q^2,q^3}{q^9}}\Sigma(4,4,9) 
	+\frac{q^{12}}{\jacprod{q,q^2,q^3}{q^9}}\Sigma(7,13,9)
.
\end{align}
In Lemma \ref{ChanLemma1} we use $q\mapsto q^9$, $b_1=q^4, b_3 = q^8, b_4=q^7$ to get
\begin{align*}
	\frac{\aqprod{q^9}{q^9}{\infty}^2}{\jacprod{q^4,q^{-4},q^7,q^8}{q^9}}
	=&
		\frac{1}{\jacprod{q^{-8},q^{4},q^3}{q^9}}\SSeries{4}{1}{9}
		-
		\frac{q^4}{\jacprod{q^{8},q^{12},q^{11}}{q^9}}\SSeries{4}{4}{9}
		\\&
		+
		\frac{1}{\jacprod{q^{-4},q^{-12},q^{-1}}{q^9}}\SSeries{8}{17}{9}
		+
		\frac{1}{\jacprod{q^{-3},q^{-11},q}{q^9}}\SSeries{7}{13}{9}
.
\end{align*}
Simplifying the products with Proposition \ref{PropositionMiscIdents} yields 
(\ref{3DissectionEq2a}).
\end{proof}

\begin{proposition}\label{Prop2Zeta3}
\begin{align}
	&q^{11}\Sigma(7,10,9)	+q^{18}\Sigma(7,19,9)
	\nonumber\\\label{3DissectionEq3}
	&=	
	\frac{\aqprod{q^{9}}{q^{9}}{\infty}^2\jacprod{q^3}{q^9}\jacprod{q^4}{q^9}}
	{\jacprod{q}{q^{9}}\jacprod{q^2}{q^{9}}^2}
	-
	\frac{\jacprod{q^4}{q^9}}{\jacprod{q^2}{q^9}}
	\Parans{\Sigma(1,-14,9)+q\Sigma(1,-5,9)}
	,\\
	&q^3\Sigma(4,-2,9)+q^{7}\Sigma(4,7,9)
	\nonumber\\\label{3DissectionEq4}
	&=
	-\frac{\aqprod{q^{9}}{q^{9}}{\infty}^2\jacprod{q^3}{q^9}}
	{\jacprod{q}{q^{9}}\jacprod{q^4}{q^{9}}}
	+
	\frac{\jacprod{q}{q^9}}{\jacprod{q^2}{q^9}}
	\Parans{\Sigma(1,-14,9) + q\Sigma(1,-5,9)}
.
\end{align}
\end{proposition}
\begin{proof}
For (\ref{3DissectionEq3}), we move all $\SSeries{z}{w}{q}$ terms to one side
and multiply by $\jacprod{q^2}{q^9}$, so that
(\ref{3DissectionEq3}) is equivalent to
\begin{align}
	\label{3DissectionEq3a}
	&\frac{\aqprod{q^9}{q^9}{\infty}^2\jacprod{q^3,q^4}{q^9}}{\jacprod{q,q^2}{q^9}}		
	\nonumber\\
	&=
	\jacprod{q^2}{q^9}\Parans{q^{11}\SSeries{7}{10}{9} + q^{18}\SSeries{7}{19}{9}}
	+\jacprod{q^4}{q^9}\Parans{\SSeries{1}{-14}{9} + q\SSeries{1}{-5}{9}}
.
\end{align}
However by (\ref{SigmaProperty4}) and the definition of $j(z;q)$, we
find that the right hand side of (\ref{3DissectionEq3a}) is
\begin{align*}
	&
		\jacprod{q^2}{q^9}\Parans{q^{4}\SSeries{7}{1}{9} 
			+ q^{25}\SSeries{7}{28}{9} - q^{4}\J{q^{-5}}{q^9}}
		\\&
		+\jacprod{q^4}{q^9}\Parans{q^{-1}\SSeries{1}{-23}{9} 
			+ q^{2}\SSeries{1}{4}{9} - q^{-1}\J{q^7}{q^9}}
	\\
	&=
	\jacprod{q^2}{q^9}\Parans{q^{4}\SSeries{7}{1}{9} 
		+ q^{25}\SSeries{7}{28}{9}}
	+\jacprod{q^4}{q^9}\Parans{q^{-1}\SSeries{1}{-23}{9} 
		+ q^{2}\SSeries{1}{4}{9}}
.
\end{align*}
With the above and dividing (\ref{3DissectionEq3a}) by
$\jacprod{q,q^2,q^3,q^4}{q^9}$,
we have that (\ref{3DissectionEq3}) is equivalent to
\begin{align}
	\label{3DissectionEq3b}
	\frac{\aqprod{q^9}{q^9}{\infty}^2}{\jacprod{q,q,q^2,q^2}{q^9}}
	&=
	\frac{q^{4}}{\jacprod{q,q^3,q^4}{q^9}}\SSeries{7}{1}{9} 
	+\frac{q^{25}}{\jacprod{q,q^3,q^4}{q^9}}\SSeries{7}{28}{9}
	\nonumber\\&\quad
	+\frac{q^{-1}}{\jacprod{q,q^2,q^3}{q^9}}\SSeries{1}{-23}{9} 
	+\frac{q^{2}}{\jacprod{q,q^2,q^3}{q^9}}\SSeries{1}{4}{9}
.
\end{align}
Setting $q\mapsto q^9, b_1 = q^7, b_3 = q$ in Lemma \ref{ChanLemma2} gives
\begin{align*}
	\frac{\aqprod{q^9}{q^9}{\infty}^2}{\jacprod{q,q^{-1},q^7,q^{-7}}{q^9}}
	=&
		\frac{1}{\jacprod{q^{-14},q^{-6},q^{-8}}{q^9}}\SSeries{7}{28}{9}
		-
		\frac{q^7}{\jacprod{q^{14},q^{8},q^{6}}{q^9}}\SSeries{7}{1}{9}
		\\&
		+
		\frac{1}{\jacprod{q^{6},q^{-8},q^{-2}}{q^9}}\SSeries{1}{4}{9}
		-
		\frac{q}{\jacprod{q^{8},q^{-6},q^{2}}{q^9}}\SSeries{1}{-23}{9}
.
\end{align*}
Simplifying the products gives (\ref{3DissectionEq3b}).

For (\ref{3DissectionEq4}), we move all $\SSeries{z}{w}{q}$ terms to one side
and multiply by $\jacprod{q^2}{q^9}$, so that
(\ref{3DissectionEq4}) is equivalent to
\begin{align}
	\label{3DissectionEq4a}
	&\frac{\aqprod{q^9}{q^9}{\infty}^2\jacprod{q^2,q^3}{q^9}}{\jacprod{q,q^4}{q^9}}		
	\nonumber\\
	&=
	\jacprod{q}{q^9}\Parans{\SSeries{1}{-14}{9} + q\SSeries{1}{-5}{9}}
	+\jacprod{q^2}{q^9}\Parans{-q^3\SSeries{4}{-2}{9} - q^7\SSeries{4}{7}{9}}
.
\end{align}
However, by (\ref{SigmaProperty4}),
the right hand side of (\ref{3DissectionEq4a}) is
\begin{align*}		
	&
		\jacprod{q}{q^9}\Parans{q^{-1}\SSeries{1}{-23}{9} 
			+ q^{2}\SSeries{1}{4}{9} - q^{-1}\J{q^7}{q^9}}
		\\&
		+\jacprod{q^2}{q^9}\Parans{-q^{-1}\SSeries{4}{-11}{9} 
			- q^{11}\SSeries{4}{16}{9} + q^{-1}\J{q}{q^9}}
	\\
	&=
	\jacprod{q}{q^9}\Parans{q^{-1}\SSeries{1}{-23}{9} 
		+ q^{2}\SSeries{1}{4}{9}}
	+\jacprod{q^2}{q^9}\Parans{-q^{-1}\SSeries{4}{-11}{9} 
		- q^{11}\SSeries{4}{16}{9}}
.
\end{align*}
With the above and dividing (\ref{3DissectionEq4a}) by 
$\jacprod{q,q^2,q^3,q^4}{q^9}$,
we have that (\ref{3DissectionEq4}) is equivalent to
\begin{align}\label{3DissectionEq4b}
	\frac{\aqprod{q^9}{q^9}{\infty}^2}{\jacprod{q,q,q^4,q^4}{q^9}}
	&=
	\frac{q^{-1}}{\jacprod{q^2,q^3,q^4}{q^9}}\SSeries{1}{-23}{9} 
	+\frac{q^{2}}{\jacprod{q^2,q^3,q^4}{q^9}}\SSeries{1}{4}{9}
	\nonumber\\&\quad
	-\frac{q^{-1}}{\jacprod{q,q^3,q^4}{q^9}}\SSeries{4}{-11}{9} 
	-\frac{q^{11}}{\jacprod{q,q^3,q^4}{q^9}}\SSeries{4}{16}{9}
.	
\end{align}
Setting $q\mapsto q^9, b_1 = q, b_3 = q^4$ in Lemma \ref{ChanLemma2} gives
\begin{align*}
	\frac{\aqprod{q^9}{q^9}{\infty}^2}{\jacprod{q,q^{-1},q^4,q^{-4}}{q^9}}
	=&
		\frac{1}{\jacprod{q^{-2},q^{3},q^{-5}}{q^9}}\SSeries{1}{4}{9}
		-
		\frac{q}{\jacprod{q^{2},q^{5},q^{-3}}{q^9}}\SSeries{1}{-23}{9}
		\\&
		+
		\frac{1}{\jacprod{q^{-3},q^{-5},q^{-8}}{q^9}}\SSeries{4}{16}{9}
		-
		\frac{q^4}{\jacprod{q^{5},q^{3},q^{8}}{q^9}}\SSeries{4}{-11}{9}
.
\end{align*}
Simplifying the products gives (\ref{3DissectionEq4b}).
\end{proof}

Next we make use of
\begin{align}\label{3DissectionOfEta}
	\aqprod{q}{q}{\infty}
	&=
	\aqprod{q^{27}}{q^{27}}{\infty}
	\Parans{\jacprod{q^{12}}{q^{27}} - q\jacprod{q^{6}}{q^{27}} -q^2\jacprod{q^{3}}{q^{27}}}
,
\end{align}
which follows from Euler's pentagonal number theorem and the 
Jacobi triple product identity.
\begin{proposition}\label{Prop3Zeta3}
\begin{align*}
	&\frac{\aqprod{q^{27}}{q^{27}}{\infty}^2\jacprod{q^9}{q^{27}}\jacprod{q^{12}}{q^{27}}}
	{\jacprod{q^3}{q^{27}}\jacprod{q^6}{q^{27}}^2}
	-
	2q^2\frac{\aqprod{q^{27}}{q^{27}}{\infty}^2\jacprod{q^9}{q^{27}}}
	{\jacprod{q^3}{q^{27}}\jacprod{q^{12}}{q^{27}}}
	-
	q^4\frac{\aqprod{q^{27}}{q^{27}}{\infty}^2\jacprod{q^9}{q^{27}}}
	{\jacprod{q^{12}}{q^{27}}^2}
	\\
	&=
	\frac{\aqprod{q}{q}{\infty}\aqprod{q^{27}}{q^{27}}{\infty}\jacprod{q^9}{q^{27}}}
		{\jacprod{q^3}{q^{27}}^2\jacprod{q^{12}}{q^{27}}}
	+
	q\frac{\aqprod{q}{q}{\infty}\aqprod{q^{9}}{q^{9}}{\infty}}{\jacprod{q^3}{q^9}}
.
\end{align*}
\end{proposition}
\begin{proof}
Multiplying both sides by
$\frac{\jacprod{q^3,q^{12}}{q^{27}}}{\aqprod{q^{27}}{q^{27}}{\infty}\jacprod{q^9}{q^{27}}}$,
we find this proposition is equivalent to
\begin{align}
\label{ProductMessEq1}
	\aqprod{q^{27}}{q^{27}}{\infty}\Parans{
		\frac{\jacprod{q^{12}}{q^{27}}^2}{\jacprod{q^6}{q^{27}}^2}
		-
		2q^2
		-
		q^4\frac{\jacprod{q^3}{q^{27}}}{\jacprod{q^{12}}{q^{27}}}
	}
	&=
	\aqprod{q}{q}{\infty}\Parans{
		\frac{1}{\jacprod{q^3}{q^{27}}}
		+\frac{q}{\jacprod{q^3}{q^{27}}\jacprod{q^6}{q^{27}}}
	}
.
\end{align}
By (\ref{3DissectionOfEta}) we see (\ref{ProductMessEq1}) reduces to proving
\begin{align}
	\label{ProductMessEq2}
	\frac{\jacprod{q^{12}}{q^{27}}^2}{\jacprod{q^6}{q^{27}}^2}
	&=	
		\frac{\jacprod{q^{12}}{q^{27}}}{\jacprod{q^3}{q^{27}}}
		-q^3\frac{\jacprod{q^{3}}{q^{27}}}{\jacprod{q^6}{q^{27}}}
	,\\
	\label{ProductMessEq3}
	-q^4\frac{\jacprod{q^3}{q^{27}}}{\jacprod{q^{12}}{q^{27}}}	
	&=	
		-q\frac{\jacprod{q^{6}}{q^{27}}}{\jacprod{q^3}{q^{27}}}
		+q\frac{\jacprod{q^{12}}{q^{27}}}{\jacprod{q^6}{q^{27}}}
.
\end{align}
However (\ref{ProductMessEq2}) follows from
multiplying (\ref{ProductMessEq3}) 
by $q\frac{\jacprod{q^{6}}{q^{27}}}{\jacprod{q^{12}}{q^{27}}}$ and 
elementary rearrangements. 
In (\ref{ProductMessEq3}) we
replace $q$ by $q^{1/3}$ and clear denominators to see we need only prove that
\begin{align}
	q\jacprod{q,q,q^2}{q^9}
	&=
	\jacprod{q^2,q^2,q^4}{q^9}
	-\jacprod{q,q^4,q^4}{q^9}
.
\label{ProductMessEq4}
\end{align}
However, this follows from the $q\mapsto q^9, x=q,t=y=z=q^2$ case of equation (2.1) 
from \cite{GarvanYesilyurt},
which is an identity of Jacobi:
\begin{align*}
	\frac{z}{x}\jacprod{y,x,xt/z,zty}{q}
	&=
	\jacprod{z,t,xty,zy/x}{q}
	-\jacprod{xt,zy,ty,z/x}{q}
.	
\end{align*}

\end{proof}

Letting $\zeta_3$ be a primitive third root of unity, we find that
\begin{align}
	\ST(\zeta_3,q)
	&=
	\frac{1}{\aqprod{q}{q}{\infty}}
	\sum_{n=-\infty}^\infty
	\frac{q^{6n^2+4n+1}(1-q^{6n+2})(1-q^{3n+1})}{1-q^{9n+3}}
	\nonumber\\
	&=
	\frac{1}{\aqprod{q}{q}{\infty}}
	\Parans{
	qU_3(4) - q^2U_3(7) - q^3U_3(10) + q^4U(13)
	}
\label{Theorem2Eq1}
.
\end{align}

Using this form of $\ST(\zeta_3,q)$ in terms of the
$U_3(b)$, we proceed in a manner similar to how Atkin and Swinnerton-Dyer in
\cite{AS} 
determined rank difference formulas for Dyson's rank of a partition. 
This was also used
to determine crank difference formulas by Ekin in \cite{Ekin}
and various rank difference formulas related to overpartitions by
Lovejoy and Osburn in \cite{LO1,LO2,LO3}. Here the major difference is that
in $\SSeries{z}{w}{q}$ we never use $z=1$ and so we do not have to introduce 
an extra
function to avoid the issue at $n=0$. 
To begin we note that
\begin{align*}
	U_3(b) &= \sum_{n=-\infty}^\infty \frac{q^{6n^2 + bn}}{1-q^{9n+3}}
	\\
	&= 
	\sum_{k=0}^2\sum_{n=-\infty}^\infty\frac{q^{6(3n+k)^2 + b(3n+k)}}{1-q^{9(3n+k)+3}}
	\\
	&=
	\sum_{k=0}^2 q^{6k^2+bk}
	\sum_{n=-\infty}^\infty \frac{ q^{54n(n+1)}q^{36nk+3nb-54n}}
		{1-q^{9k+3}q^{27n}}
	\\
	&=
	\sum_{k=0}^2 q^{6k^2+bk}
	\SSeries{9k+3}{36k+3b-54}{27}
.
\end{align*}
Thus
\begin{align*}
	&qU_3(4) - q^2U_3(7) - q^3U_3(10) + q^4U_3(13)
	\\
	&=
		q\Sigma(3,-42,27) 		
		+q^{11}\Sigma(12,-6,27) 		
		+q^{33}\Sigma(21,30,27)
		-q^2\Sigma(3,-33,27)  
		-q^{15}\Sigma(12,3,27)
		\\&\quad
		- q^{40}\Sigma(21,39,27)
		-q^3\Sigma(3,-24,27)
		- q^{19}\Sigma(12,12,27) 
		- q^{47}\Sigma(21,48,27)
		+q^4\Sigma(3,-15,27)  
		\\&\quad
		+q^{54}\Sigma(21,57,27)
		+q^{23}\Sigma(12,21,27)
	\\
	&=
		q\Sigma(3,-42,27) 		
		+q^4\Sigma(3,-15,27)  
		-q^3\Sigma(3,-24,27)
		-q^{15}\Sigma(12,3,27)
		\\&\quad
		+q^{33}\Sigma(21,30,27)
		+q^{54}\Sigma(21,57,27)
		+q^{23}\Sigma(12,21,27)
		+q^{11}\Sigma(12,-6,27) 		
		\\&\quad
		-q^2\Sigma(3,-33,27)  
		- q^{47}\Sigma(21,48,27)
		- q^{19}\Sigma(12,12,27) 
		- q^{40}\Sigma(21,39,27)
.
\end{align*}
In the last line we have ordered the terms to apply 
Propositions \ref{Prop1Zeta3} and \ref{Prop2Zeta3}
with $q\mapsto q^3$.
With these identities we have that
\begin{align*}
	&qU_3(4) - q^2U_3(7) - q^3U_3(10) + q^4U_3(13)
	\\
	&=
		\Parans{\Sigma(3,-42,27) +q^3\Sigma(3,-15,27)}
		\Parans{q + q^2\frac{\jacprod{q^3}{q^{27}}}{\jacprod{q^6}{q^{27}}} 
					 - \frac{\jacprod{q^{12}}{q^{27}}}{\jacprod{q^6}{q^{27}}}  
			}
		\\&\quad
		-
		\Parans{q^3\Sigma(3,-24,27) +q^{15}\Sigma(12,3,27)}
		\Parans{1 - q\frac{\jacprod{q^6}{q^{27}}}{\jacprod{q^{12}}{q^{27}}} 
					 - q^2\frac{\jacprod{q^3}{q^{27}}}{\jacprod{q^{12}}{q^{27}}}  
			}	
		\\&\quad
		+\frac{\aqprod{q^{27}}{q^{27}}{\infty}^2\jacprod{q^9,q^{12}}{q^{27}}}
		{\jacprod{q^3,q^6,q^6}{q^{27}}}
		-
		2q^2\frac{\aqprod{q^{27}}{q^{27}}{\infty}^2\jacprod{q^9}{q^{27}}}
		{\jacprod{q^3,q^{12}}{q^{27}}}
		-
		q^4\frac{\aqprod{q^{27}}{q^{27}}{\infty}^2\jacprod{q^9}{q^{27}}}
		{\jacprod{q^{12},q^{12}}{q^{27}}}
		.
\end{align*}
Next by (\ref{3DissectionOfEta}) and Proposition \ref{Prop3Zeta3}
we have that
\begin{align}
	&qU_3(4) - q^2U_3(7) - q^3U_3(10) + q^4U_3(13)
	\nonumber\\
	&=
		-
		\frac{\aqprod{q}{q}{\infty}}{\aqprod{q^{27}}{q^{27}}{\infty}\jacprod{q^6}{q^{27}}}
		\Parans{\Sigma(3,-42,27) +q^3\Sigma(3,-15,27)}
		\nonumber\\&\quad
		-
		\frac{\aqprod{q}{q}{\infty}}{\aqprod{q^{27}}{q^{27}}{\infty}\jacprod{q^{12}}{q^{27}}}
		\Parans{q^3\Sigma(3,-24,27) +q^{15}\Sigma(12,3,27)}
		\nonumber\\&\quad
		+
		\frac{\aqprod{q}{q}{\infty}\aqprod{q^{27}}{q^{27}}{\infty}\jacprod{q^9}{q^{27}}}
			{\jacprod{q^3}{q^{27}}^2\jacprod{q^{12}}{q^{27}}}
		+
		q\frac{\aqprod{q}{q}{\infty}\aqprod{q^{9}}{q^{9}}{\infty}}{\jacprod{q^3}{q^9}}
.
\label{Theorem2Eq2}
\end{align}
Theorem \ref{theorem2} now follows by equations (\ref{Theorem2Eq1}) and
(\ref{Theorem2Eq2}).

\section{Proof of Theorem \ref{theorem3}}

We require the following two Lemmas. Both are applications of Theorem 2.1 of 
\cite{Chan},
namely we take $s=6$, $r=2$, $b_4=b_1^{-1}$,
$b_5=b_2^{-1}$, $b_6=b_3^{-1}$ to obtain Lemma \ref{ChanLemma3} and
take 
$s=10$, $r=6$, $b_6=b_1^{-1}$,
$b_7=b_2^{-1}$, $b_8=b_3^{-1}$, $b_9=b_4^{-1}$, and $b_{10}=b_5^{-1}$ to obtain
Lemma \ref{ChanLemma4}.
\begin{lemma}\label{ChanLemma3}
\begin{align*}
	\frac{\jacprod{a_1,a_2}{q}\aqprod{q}{q}{\infty}^2}
		{\jacprod{b_1,b_1^{-1},b_2,b_2^{-1},b_3,b_3^{-1}}{q}}
	=&
		\frac{\jacprod{a_1b_1^{-1},a_2b_1^{-1}}{q}}
			{\jacprod{b_1^{-1}b_2,b_1^{-1}b_3,b_1^{-1}b_2^{-1},b_1^{-1}b_3^{-1},b_1^{-2}}{q}}
			\SSeries{b_1}{a_1a_2b_1^4}{q}
		\\&
		-
		\frac{b_1\jacprod{a_1b_1,a_2b_1}{q}}
			{\jacprod{b_1b_2,b_1b_3,b_1b_2^{-1},b_1b_3^{-1},b_1^{2}}{q}}
			\SSeries{b_1}{\frac{b_1^4}{a_1a_2q^3}}{q}
		\\&
		+\frac{\jacprod{a_1b_2^{-1},a_2b_2^{-1}}{q}}
			{\jacprod{b_1b_2^{-1},b_2^{-1}b_3,b_1^{-1}b_2^{-1},b_2^{-1}b_3^{-1},b_2^{-2}}{q}}
			\SSeries{b_2}{a_1a_2b_2^4}{q}
		\\&
		-\frac{b_2\jacprod{a_1b_2,a_2b_2}{q}}
			{\jacprod{b_1b_2,b_2b_3,b_1^{-1}b_2,b_2b_3^{-1},b_2^{2}}{q}}
			\SSeries{b_2}{\frac{b_2^4}{a_1a_2q^3}}{q}
		\\&
		+\frac{\jacprod{a_1b_3^{-1},a_2b_3^{-1}}{q}}
			{\jacprod{b_1b_3^{-1},b_2b_3^{-1},b_1^{-1}b_3^{-1},b_2^{-1}b_3^{-1},b_3^{-2}}{q}}
			\SSeries{b_3}{a_1a_2b_3^4}{q}
		\\&
		-\frac{b_3\jacprod{a_1b_3,a_2b_3}{q}}
			{\jacprod{b_1b_3,b_2b_3,b_1^{-1}b_3,b_2^{-1}b_3,b_3^{2}}{q}}
			\SSeries{b_3}{\frac{b_3^4}{a_1a_2q^3}}{q}
.
\end{align*}
\end{lemma}

\begin{lemma}\label{ChanLemma4}
\begin{align*}
	&\frac{\jacprod{a_1,a_2,a_3,a_4,a_5,a_6}{q}\aqprod{q}{q}{\infty}^2}
		{\jacprod{b_1,b_1^{-1},b_2,b_2^{-1},b_3,b_3^{-1},b_4,b_4^{-1},b_5,b_5^{-1}}{q}}
	\\
	&=
		\frac{\jacprod{a_1b_1^{-1},a_2b_1^{-1},a_3b_1^{-1},a_4b_1^{-1},a_5b_1^{-1},a_6b_1^{-1}}{q}}
		{\jacprod{b_1^{-1}b_2, b_1^{-1}b_3, b_1^{-1}b_4, b_1^{-1}b_5, 
			b_1^{-1}b_2^{-1}, b_1^{-1}b_3^{-1}, b_1^{-1}b_4^{-1}, b_1^{-1}b_5^{-1}, b_1^{-2}}{q}}
		\SSeries{b_1}{a_1a_2a_3a_4a_5a_6b_1^4}{q}
		\\&\quad
		-
		\frac{b_1\jacprod{a_1b_1,a_2b_1,a_3b_1,a_4b_1,a_5b_1,a_6b_1}{q}}
		{\jacprod{b_1b_2, b_1b_3, b_1b_4, b_1b_5,   
			b_1b_2^{-1}, b_1b_3^{-1}, b_1b_4^{-1}, b_1b_5^{-1},  b_1^{2}}{q}}
			\SSeries{b_1}{\frac{b_1^4}{a_1a_2a_3a_4a_5a_6q^3}}{q}
		\\&\quad
		+\frac{\jacprod{a_1b_2^{-1},a_2b_2^{-1},a_3b_2^{-1},a_4b_2^{-1},a_5b_2^{-1},a_6b_2^{-1}}{q}}
		{\jacprod{b_1b_2^{-1}, b_2^{-1}b_3, b_2^{-1}b_4, b_2^{-1}b_5, 
			b_1^{-1}b_2^{-1}, b_2^{-1}b_3^{-1}, b_2^{-1}b_4^{-1}, b_2^{-1}b_5^{-1}, b_2^{-2}}{q}}
		\SSeries{b_2}{a_1a_2a_3a_4a_5a_6b_3^4}{q}
		\\&\quad
		-
		\frac{b_2\jacprod{a_1b_2,a_2b_2,a_3b_2,a_4b_2,a_5b_2,a_6b_2}{q}}
		{\jacprod{b_1b_2, b_2b_3, b_2b_4, b_2b_5,   
			b_1^{-1}b_2, b_2b_3^{-1}, b_2b_4^{-1}, b_2b_5^{-1},  b_2^{2}}{q}}
			\SSeries{b_2}{\frac{b_2^4}{a_1a_2a_3a_4a_5a_6q^3}}{q}
		\\&\quad
		+\frac{\jacprod{a_1b_3^{-1},a_2b_3^{-1},a_3b_3^{-1},a_4b_3^{-1},a_5b_3^{-1},a_6b_3^{-1}}{q}}
		{\jacprod{b_1b_3^{-1}, b_2b_3^{-1}, b_3^{-1}b_4, b_3^{-1}b_5, 
			b_1^{-1}b_3^{-1}, b_2^{-1}b_3^{-1}, b_3^{-1}b_4^{-1}, b_3^{-1}b_5^{-1}, b_3^{-2}}{q}}
		\SSeries{b_3}{a_1a_2a_3a_4a_5a_6b_3^4}{q}
		\\&\quad
		-
		\frac{b_3\jacprod{a_1b_3,a_2b_3,a_3b_3,a_4b_3,a_5b_3,a_6b_3}{q}}
		{\jacprod{b_1b_3, b_2b_3, b_3b_4, b_3b_5,   
			b_1^{-1}b_3, b_2^{-1}b_3, b_3b_4^{-1}, b_3b_5^{-1}, b_3^{2}}{q}}
		\SSeries{b_3}{\frac{b_3^4}{a_1a_2a_3a_4a_5a_6q^3}}{q}
		\\&\quad
		+\frac{\jacprod{a_1b_4^{-1},a_2b_4^{-1},a_3b_4^{-1},a_4b_4^{-1},a_5b_4^{-1},a_6b_4^{-1}}{q}}
		{\jacprod{b_1b_4^{-1}, b_2b_4^{-1}, b_3b_4^{-1}, b_4^{-1}b_5, 
			b_1^{-1}b_4^{-1}, b_2^{-1}b_4^{-1}, b_3^{-1}b_4^{-1}, b_4^{-1}b_5^{-1}, b_4^{-2}}{q}}
		\SSeries{b_4}{a_1a_2a_3a_4a_5a_6b_4^4}{q}
		\\&\quad
		-
		\frac{b_4\jacprod{a_1b_4,a_2b_4,a_3b_4,a_4b_4,a_5b_4,a_6b_4}{q}}
		{\jacprod{b_1b_4, b_2b_4, b_3b_4, b_4b_5,   
			b_1^{-1}b_4, b_2^{-1}b_4, b_3^{-1}b_4, b_4b_5^{-1}, b_4^{2}}{q}}
		\SSeries{b_4}{\frac{b_4^4}{a_1a_2a_3a_4a_5a_6q^3}}{q}
		\\&\quad
		+\frac{\jacprod{a_1b_5^{-1},a_2b_5^{-1},a_3b_5^{-1},a_4b_5^{-1},a_5b_5^{-1},a_6b_5^{-1}}{q}}
		{\jacprod{b_1b_5^{-1}, b_2b_5^{-1}, b_3b_5^{-1}, b_4b_5^{-1}, 
			b_1^{-1}b_5^{-1}, b_2^{-1}b_5^{-1}, b_3^{-1}b_5^{-1}, b_4^{-1}b_5^{-1}, b_5^{-2}}{q}}
		\SSeries{b_5}{a_1a_2a_3a_4a_5a_6b_5^4}{q}
		\\&\quad
		-
		\frac{b_5\jacprod{a_1b_5,a_2b_5,a_3b_5,a_4b_5,a_5b_5,a_6b_5}{q}}
		{\jacprod{b_1b_5, b_2b_5, b_3b_5, b_4b_5,   
			b_1^{-1}b_5, b_2^{-1}b_5, b_3^{-1}b_5, b_4^{-1}b_5, b_5^{2}}{q}}
		\SSeries{b_5}{\frac{b_5^4}{a_1a_2a_3a_4a_5a_6q^3}}{q}
.
\end{align*}
\end{lemma}

\sloppy
In the following Propositions we will repeatedly use that 
$\jacprod{q,q^4,q^6}{q^{15}} = \jacprod{q}{q^5}$ and 
$\jacprod{q^2,q^3,q^7}{q^{15}} = \jacprod{q^2}{q^5}$ 
in reducing the products.
The proofs of these Propositions are similar to the proofs of Propositions
\ref{Prop1Zeta3} and \ref{Prop2Zeta3}. By moving the
$\SSeries{z}{w}{q}$ terms all to one side
and multiplying by the appropriate
product, we find the identities are
equivalent to a specialization of Lemma \ref{ChanLemma3} or
\ref{ChanLemma4}.

\fussy

\begin{proposition}\label{Prop1Zeta5}
\begin{align}
	&\SSeries{7}{-2}{15}
	+q^7\SSeries{7}{13}{15}
	+q^{16}\SSeries{13}{22}{15}
	+q^{29}\SSeries{13}{37}{15}
	\nonumber\\\label{5DissectionEq1}	
	&=
	\frac{\jacprod{q}{q^{5}}}
	{\jacprod{q^{2}}{q^{5}}}	
	\Parans{q^{7}\SSeries{10}{10}{15}+q^{17}\SSeries{10}{25}{15}}
	-
	q^{-6}\frac{\aqprod{q^{3}}{q^{3}}{\infty}^3}
		{\aqprod{q^5}{q^5}{\infty}\jacprod{q^{2}}{q^{5}}^2}
	,\\
	&\SSeries{1}{-26}{15} + q\SSeries{1}{-11}{15} + q^{2}\SSeries{4}{-14}{15}
	+q^{6}\SSeries{4}{1}{15}
	\nonumber\\\label{5DissectionEq2}
	&=
		-\frac{\jacprod{q^{2}}{q^{5}}}{\jacprod{q}{q^{5}}}
		\Parans{q^{13}\SSeries{10}{10}{15}+q^{23}\SSeries{10}{25}{15}}
		+
		\frac{\aqprod{q^3}{q^3}{\infty}^3}
			{\aqprod{q^5}{q^5}{\infty}\jacprod{q}{q^{15}}^2}
.
\end{align}
\end{proposition}
\begin{proof}
In Lemma \ref{ChanLemma3} we use $q\mapsto q^{15}$, $a_1=q^{-9}$, $a_2=q^{-21}$,
$b_1=q^7$, $b_2=q^{10}$, $b_3=q^{13}$ to get

\begin{align*}
	&\frac{\jacprod{q^{-9},q^{-21}}{q^{15}}\aqprod{q^{15}}{q^{15}}{\infty}^2}
		{\jacprod{q^{7},q^{10},q^{13},q^{-7},q^{-10},q^{-13}}{q^{15}}}
	\\
	&= 
		\frac{\jacprod{q^{-16},q^{-28}}{q^{15}}}
			{\jacprod{q^{3},q^{6},q^{-14},q^{-17},q^{-20}}{q^{15}}}
		\SSeries{7}{-2}{15}
		-
		q^{7}\frac{\jacprod{q^{-2},q^{-14}}{q^{15}}}
			{\jacprod{q^{14},q^{17},q^{20},q^{-3},q^{-6}}{q^{15}}}
		\SSeries{7}{13}{15}
		\\&\quad
		+
		\frac{\jacprod{q^{-19},q^{-31}}{q^{15}}}
			{\jacprod{q^{-3},q^{3},q^{-17},q^{-20},q^{-23}}{q^{15}}}
		\SSeries{10}{10}{15}
		-
		q^{10}\frac{\jacprod{q^{-11}}{q^{15}}}
			{\jacprod{q^{17},q^{20},q^{23},q^{3},q^{-3}}{q^{15}}}
		\SSeries{10}{25}{15}
		\\&\quad
		+
		\frac{\jacprod{q^{-22},q^{-34}}{q^{15}}}
			{\jacprod{q^{-6},q^{-3},q^{-20},q^{-23},q^{-26}}{q^{15}}}
		\SSeries{13}{22}{15}
		-
		q^{13}\frac{\jacprod{q^{4},q^{-8}}{q^{15}}}
			{\jacprod{q^{20},q^{23},q^{26},q^{6},q^{3}}{q^{15}}}
		\SSeries{13}{37}{15}
.
\end{align*}
Simplifying the products yields
\begin{align}
	\label{5DissectionEq1a}
	&q^{-6}\frac{\jacprod{q^{6},q^{6}}{q^{15}}\aqprod{q^{15}}{q^{15}}{\infty}^2}
		{\jacprod{q^{2},q^{2},q^{5},q^{5},q^{7},q^{7}}{q^{15}}}
	\nonumber\\
	&=
		-\frac{1}{\jacprod{q^{3},q^{5},q^{6}}{q^{15}}}\SSeries{7}{-2}{15}
		-
		q^7\frac{1}{\jacprod{q^{3},q^{5},q^{6}}{q^{15}}}\SSeries{7}{13}{15}
		\nonumber\\&\quad
		+
		q^{7}\frac{\jacprod{q,q^{4}}{q^{15}}}
			{\jacprod{q^{2},q^{3},q^{3},q^{5},q^{7}}{q^{15}}}
		\SSeries{10}{10}{15}
		+
		q^{17}\frac{\jacprod{q,q^{4}}{q^{15}}}
			{\jacprod{q^{2},q^{3},q^{3},q^{7}}{q^{15}}}
		\SSeries{10}{25}{15}
		\nonumber\\&\quad
		-q^{16}\frac{1}{\jacprod{q^{3},q^{5},q^{6}}{q^{15}}}\SSeries{13}{22}{15}
		-
		q^{29}\frac{1}{\jacprod{q^{3},q^{5},q^{6}}{q^{15}}}\SSeries{13}{37}{15}
	.
\end{align}
We see that multiplying both sides of (\ref{5DissectionEq1a})
by $\jacprod{q^3,q^5,q^6}{q^{15}}$
implies (\ref{5DissectionEq1}) upon noting that
$\frac{\jacprod{q^3,q^6,q^6,q^6}{q^{15}}\aqprod{q^{15}}{q^{15}}{\infty}^2}
	{\jacprod{q^2,q^2,q^5,q^7,q^7}{q^{15}}}
= \frac{\aqprod{q^3}{q^3}{\infty}^3}{\aqprod{q^5}{q^5}{\infty}\jacprod{q^2}{q^5}^2}$
.

In Lemma \ref{ChanLemma3} we use $q\mapsto q^{15}$, $a_1=q^{-12}$, $a_2=q^{-18}$,
$b_1=q$, $b_2=q^4$, $b_3=q^{10}$ to get
\begin{align*}
	&\frac{\jacprod{q^{-12},q^{-18}}{q^{15}}\aqprod{q^{15}}{q^{15}}{\infty}^2}
	{\jacprod{q,q^{4},q^{10},q^{-1},q^{-4},q^{-10}}{q^{15}}}
	\\
	&=
		\frac{\jacprod{q^{-13},q^{-19}}{q^{15}}}
			{\jacprod{q^{3},q^{9},q^{-2},q^{-5},q^{-11}}{q^{15}}}
		\SSeries{1}{-26}{15}
		-
		q\frac{\jacprod{q^{-11},q^{-17}}{q^{15}}}
			{\jacprod{q^{2},q^{5},q^{11},q^{-3},q^{-9}}{q^{15}}}
		\SSeries{1}{-11}{15}
		\\&\quad
		+\frac{\jacprod{q^{-16},q^{-22}}{q^{15}}}
			{\jacprod{q^{-3},q^{6},q^{-5},q^{-8},q^{-14}}{q^{15}}}
		\SSeries{4}{-14}{15}
		-
		q^{4}\frac{\jacprod{q^{-8},q^{-14}}{q^{15}}}
			{\jacprod{q^{5},q^{8},q^{14},q^{3},q^{-6}}{q^{15}}}
		\SSeries{4}{1}{15}
		\\&\quad
		+\frac{\jacprod{q^{-22},q^{-28}}{q^{15}}}
			{\jacprod{q^{-9},q^{-6},q^{-11},q^{-14},q^{-20}}{q^{15}}}
		\SSeries{10}{10}{15}
		-
		q^{10}\frac{\jacprod{q^{-2},q^{-8}}{q^{15}}}
			{\jacprod{q^{11},q^{14},q^{20},q^{9},q^{6}}{q^{15}}}
		\SSeries{10}{25}{15}
.
\end{align*}
Simplifying the products yields
\begin{align}
	\label{5DissectionEq2a}
	&q^{-18}\frac{\jacprod{q^{3},q^{3}}{q^{15}}\aqprod{q^{15}}{q^{15}}{\infty}^2}
	{\jacprod{q,q,q^{4},q^{4},q^{5},q^{5}}{q^{15}}}
	\nonumber\\
	&=
		q^{-18}\frac{1}{\jacprod{q^{3},q^{5},q^{6}}{q^{15}}}\SSeries{1}{-26}{15}
		+
		q^{-17}\frac{1}{\jacprod{q^{3},q^{5},q^{6}}{q^{15}}}\SSeries{1}{-11}{15}
		\nonumber\\&\quad
		+
		q^{-16}\frac{1}{\jacprod{q^{3},q^{5},q^{6}}{q^{15}}}\SSeries{4}{-14}{15}
		+
		q^{-12}\frac{1}{\jacprod{q^{3},q^{5},q^{6}}{q^{15}}}\SSeries{4}{1}{15}
		\nonumber\\&\quad
		+
		q^{-5}\frac{\jacprod{q^{2},q^{7}}{q^{15}}}
			{\jacprod{q,q^{4},q^{5},q^{6},q^{6}}{q^{15}}}
		\SSeries{10}{10}{15}
		+
		q^5\frac{\jacprod{q^{2},q^{7}}{q^{15}}}
			{\jacprod{q,q^{4},q^{5},q^{6},q^{6}}{q^{15}}}
		\SSeries{10}{25}{15}
	.
\end{align}
We see that multiplying both sides of (\ref{5DissectionEq2a})
by $q^{18}\jacprod{q^3,q^5,q^6}{q^{15}}$
implies (\ref{5DissectionEq1}) upon noting that
$\frac{\jacprod{q^3,q^3,q^3,q^6}{q^{15}}\aqprod{q^{15}}{q^{15}}{\infty}^2}
	{\jacprod{q,q,q^4,q^4,q^5}{q^{15}}}
= \frac{\aqprod{q^3}{q^3}{\infty}^3}{\aqprod{q^5}{q^5}{\infty}\jacprod{q}{q^5}^2}$
.
\end{proof}

\begin{proposition}\label{Prop2Zeta5}
\begin{align}
	&\SSeries{1}{-17}{15}
	+q^3\SSeries{4}{-8}{15}
	+q^{10}\SSeries{7}{7}{15}
	+q^{27}\SSeries{13}{28}{15}
	\nonumber\\	\label{5DissectionEq3}
	&=
	\frac{\jacprod{q}{q^{5}}}{\jacprod{q^{2}}{q^{5}}}
	\Parans{\SSeries{1}{-20}{15} +q^{4}\SSeries{4}{-5}{15}}
	,\\
	&\SSeries{7}{4}{15}+q^{8}\SSeries{10}{16}{15}+q^{10}\SSeries{10}{19}{15}
		+q^{21}\SSeries{13}{31}{15}
	\nonumber\\\label{5DissectionEq4}
	&=
		-\frac{\jacprod{q^{2}}{q^{5}}}{\jacprod{q}{q^{5}}}
		\Parans{q^{-9}\SSeries{1}{-20}{15}+q^{-5}\SSeries{4}{-5}{15}}		
		+
		q^{-9}\frac{\aqprod{q^{3}}{q^{3}}{\infty}^3}
		{\aqprod{q}{q}{\infty}}
.
\end{align}
\end{proposition}
\begin{proof}
In Lemma \ref{ChanLemma4} we use $q\mapsto q^{15}$, $a_1=q^{-15}$, 
$a_2 = q^{-13}$, $a_3 = q^{-10}$, $a_4 = q^{-8}$, $a_5 = q^{10}$, $a_6 = q^{12}$,
$b_1 = q$, $b_2 = q^{4}$, $b_3 = q^{7}$, $b_4 = q^{10}$, $b_5 = q^{13}$
and note both the product on the left and four terms on the right in 
Lemma \ref{ChanLemma4} are immediately zero, yielding
\begin{align*}
	0
	=&
		\frac{\jacprod{q^{-16},q^{-14},q^{-11},q^{-9},q^{9},q^{11}}{q^{15}}}
			{\jacprod{q^{3},q^{6},q^{9},q^{12},q^{-2},q^{-5},q^{-8},q^{-11},q^{-14}}{q^{15}}}
		\SSeries{1}{-20}{15}
		\\&
		-q\frac{\jacprod{q^{-14},q^{-12},q^{-9},q^{-7},q^{11},q^{13}}{q^{15}}}
			{\jacprod{q^{2},q^{5},q^{8},q^{11},q^{14},q^{-3},q^{-6},q^{-9},q^{-12}}{q^{15}}}
		\SSeries{1}{-17}{15}
		\\&
		+\frac{\jacprod{q^{-19},q^{-17},q^{-14},q^{-12},q^{6},q^{8}}{q^{15}}}
			{\jacprod{q^{-3},q^{3},q^{6},q^{9},q^{-5},q^{-8},q^{-11},q^{-14},q^{-17}}{q^{15}}}
		\SSeries{4}{-8}{15}
		\\&
		-q^{4}\frac{\jacprod{q^{-11},q^{-9},q^{-6},q^{-4},q^{14},q^{16}}{q^{15}}}
			{\jacprod{q^{5},q^{8},q^{11},q^{14},q^{17},q^{3},q^{-3},q^{-6},q^{-9}}{q^{15}}}
		\SSeries{4}{-5}{15}
		\\&
		-q^{7}\frac{\jacprod{q^{-8},q^{-6},q^{-3},q^{-1},q^{17},q^{19}}{q^{15}}}
			{\jacprod{q^{8},q^{11},q^{14},q^{17},q^{20},q^{6},q^{3},q^{-3},q^{-6}}{q^{15}}}
		\SSeries{7}{7}{15}
		\\&
		+\frac{\jacprod{q^{-28},q^{-26},q^{-23},q^{-21},q^{-3},q^{-1}}{q^{15}}}
			{\jacprod{q^{-12},q^{-9},q^{-6},q^{-3},q^{-14},q^{-17},q^{-20},q^{-23},q^{-26}}{q^{15}}}
		\SSeries{13}{28}{15}
.
\end{align*}
Simplifying the products yields
\begin{align}
	\label{5DissectionEq3a}
	0
	=&
		q^{-11}\frac{\jacprod{q,q^{4}}{q^{15}}}
			{\jacprod{q^{2},q^{3},q^{3},q^{5},q^{7}}{q^{15}}}\SSeries{1}{-20}{15}
		-
		q^{-11}\frac{1}{\jacprod{q^{3},q^{5},q^{6}}{q^{15}}}\SSeries{1}{-17}{15}
		\nonumber\\&
		-
		q^{-8}\frac{1}{\jacprod{q^{3},q^{5},q^{6}}{q^{15}}}\SSeries{4}{-8}{15}
		+
		q^{-7}\frac{\jacprod{q,q^{4}}{q^{15}}}
			{\jacprod{q^{2},q^{3},q^{3},q^{5},q^{7}}{q^{15}}}\SSeries{4}{-5}{15}
		\nonumber\\&
		-
		q^{-1}\frac{1}{\jacprod{q^{3},q^{5},q^{6}}{q^{15}}}\SSeries{7}{7}{15}
		-
		q^{16}\frac{1}{\jacprod{q^{3},q^{5},q^{6}}{q^{15}}}\SSeries{13}{28}{15}
.
\end{align}
We see multiplying both sides of 
(\ref{5DissectionEq3a}) by
$q^{11}\jacprod{q^3,q^5,q^6}{q^{15}}$
implies (\ref{5DissectionEq3}).


In Lemma \ref{ChanLemma4} we use $q\mapsto q^{15}$, $a_1=q^{-12}$, 
$a_2 = q^{-11}$, $a_3 = q^{-7}$, $a_4 = q^{-6}$, $a_5 = q^{-2}$, $a_6 = q^{14}$,
$b_1 = q$, $b_2 = q^{4}$, $b_3 = q^{7}$, $b_4 = q^{10}$, $b_5 = q^{13}$
and note four terms on the right in 
Lemma \ref{ChanLemma4} are immediately zero, yielding
\begin{align*}
	&\frac{\jacprod{q^{-12},q^{-11},q^{-7},q^{-6},q^{-2},q^{14}}{q^{15}}\aqprod{q^{15}}{q^{15}}{\infty}}
	{\jacprod{q^{1},q^{4},q^{7},q^{10},q^{13},q^{-1},q^{-4},q^{-7},q^{-10},q^{-13}}{q^{15}}}
	\\
	&=
		\frac{\jacprod{q^{-13},q^{-12},q^{-8},q^{-7},q^{-3},q^{13}}{q^{15}}}
			{\jacprod{q^{3},q^{6},q^{9},q^{12},q^{-2},q^{-5},q^{-8},q^{-11},q^{-14}}{q^{15}}}
		\SSeries{1}{-20}{15}
		\\&\quad
		-q^{4}\frac{\jacprod{q^{-8},q^{-7},q^{-3},q^{-2},q^{2},q^{18}}{q^{15}}}
			{\jacprod{q^{5},q^{8},q^{11},q^{14},q^{17},q^{3},q^{-3},q^{-6},q^{-9}}{q^{15}}}
		\SSeries{4}{-5}{15}
		\\&\quad
		+\frac{\jacprod{q^{-19},q^{-18},q^{-14},q^{-13},q^{-9},q^{7}}{q^{15}}}
			{\jacprod{q^{-6},q^{-3},q^{3},q^{6},q^{-8},q^{-11},q^{-14},q^{-17},q^{-20}}{q^{15}}}
		\SSeries{7}{4}{15}
		\\&\quad
		+\frac{\jacprod{q^{-22},q^{-21},q^{-17},q^{-16},q^{-12},q^{4}}{q^{15}}}
			{\jacprod{q^{-9},q^{-6},q^{-3},q^{3},q^{-11},q^{-14},q^{-17},q^{-20},q^{-23}}{q^{15}}}
		\SSeries{10}{16}{15}
		\\&\quad
		-q^{10}\frac{\jacprod{q^{-2},q^{-1},q^{3},q^{4},q^{8},q^{24}}{q^{15}}}
			{\jacprod{q^{11},q^{14},q^{17},q^{20},q^{23},q^{9},q^{6},q^{3},q^{-3}}{q^{15}}}
		\SSeries{10}{19}{15}
		\\&\quad
		-q^{13}\frac{\jacprod{q,q^{2},q^{6},q^{7},q^{11},q^{27}}{q^{15}}}
			{\jacprod{q^{14},q^{17},q^{20},q^{23},q^{26},q^{12},q^{9},q^{6},q^{3}}{q^{15}}}
		\SSeries{13}{31}{15}
.
\end{align*}
Simplifying the products yields
\begin{align}
	\label{5DissectionEq4a}
	&q^{-3}\frac{\jacprod{q^{3},q^{6}}{q^{15}}\aqprod{q^{15}}{q^{15}}{\infty}^2}
		{\jacprod{q,q^{2},q^{4},q^{5},q^{5},q^{7}}{q^{15}}}
	\nonumber\\
	&=	
		q^{-3}\frac{\jacprod{q^{2},q^{7}}{q^{15}}}
			{\jacprod{q^{1},q^{4},q^{5},q^{6},q^{6}}{q^{15}}}\SSeries{1}{-20}{15}
		+
		q\frac{\jacprod{q^{2},q^{7}}{q^{15}}}
			{\jacprod{q^{1},q^{4},q^{5},q^{6},q^{6}}{q^{15}}}\SSeries{4}{-5}{15}
		\nonumber\\&\quad
		+
		q^6\frac{1}{\jacprod{q^{3},q^{5},q^{6}}{q^{15}}}\SSeries{7}{4}{15}
		+
		q^{14}\frac{1}{\jacprod{q^{3},q^{5},q^{6}}{q^{15}}}\SSeries{10}{16}{15}
		\nonumber\\&\quad
		+
		q^{16}\frac{1}{\jacprod{q^{3},q^{5},q^{6}}{q^{15}}}\SSeries{10}{19}{15}
		+
		q^{27}\frac{1}{\jacprod{q^{3},q^{5},q^{6}}{q^{15}}}\SSeries{13}{31}{15}
.
\end{align}
\sloppy
We see multiplying (\ref{5DissectionEq4a})
by $q^{-6}\jacprod{q^3,q^5,q^6}{q^{15}}$
implies (\ref{5DissectionEq4}), upon noting that
$\frac{\jacprod{q^{3},q^3,q^6,q^{6}}{q^{15}}\aqprod{q^{15}}{q^{15}}{\infty}^2}
		{\jacprod{q^{1},q^{2},q^{4},q^{5},q^{7}}{q^{15}}}
=\frac{\aqprod{q^3}{q^3}{\infty}^3}{\aqprod{q}{q}{\infty}}
$.

\fussy

\end{proof}

\begin{proposition}\label{Prop3Zeta5}
\begin{align}
	&
	\SSeries{1}{-14}{15} + q^2\SSeries{4}{-11}{15} +q^{7}\SSeries{7}{1}{15}
		+q^{32}\SSeries{13}{34}{15}
	\nonumber\\\label{5DissectionEq5}
	&=
	-\frac{\jacprod{q^{2}}{q^{5}}}{\jacprod{q}{q^{5}}}
	\Parans{q^{11}\SSeries{7}{10}{15}+q^{24}\SSeries{13}{25}{15}}
	,\\
	&\SSeries{1}{-23}{15}+q^{5}\SSeries{4}{-2}{15}+q^{15}\SSeries{10}{13}{15}
		+q^{21}\SSeries{10}{22}{15}
	\nonumber\\\label{5DissectionEq6}
	&=
		\frac{\jacprod{q}{q^{5}}}{\jacprod{q^{2}}{q^{5}}}
		\Parans{q^{12}\SSeries{7}{10}{15}+q^{25}\SSeries{13}{25}{15}}
		+
		\frac{\aqprod{q^3}{q^3}{\infty}^3}
		{\aqprod{q}{q}{\infty}}
.
\end{align}
\end{proposition}
\begin{proof}
In Lemma \ref{ChanLemma4} we use $q\mapsto q^{15}$, $a_1=q^{-15}$, 
$a_2 = q^{-14}$, $a_3 = q^{-10}$, $a_4 = q^{-9}$, $a_5 = q^{10}$, $a_6 = q^{11}$,
$b_1 = q$, $b_2 = q^{4}$, $b_3 = q^{7}$, $b_4 = q^{10}$, $b_5 = q^{13}$
and note both the product on the left and four terms on the right in 
Lemma \ref{ChanLemma4} are immediately zero, yielding
\begin{align*}
	0
	=&
		-q\frac{\jacprod{q^{-14},q^{-13},q^{-9},q^{-8},q^{11},q^{12}}{q^{15}}}
			{\jacprod{q^{2},q^{5},q^{8},q^{11},q^{14},q^{-3},q^{-6},q^{-9},q^{-12}}{q^{15}}}
		\SSeries{1}{-14}{15}
		\\&
		+\frac{\jacprod{q^{-19},q^{-18},q^{-14},q^{-13},q^{6},q^{7}}{q^{15}}}
			{\jacprod{q^{-3},q^{3},q^{6},q^{9},q^{-5},q^{-8},q^{-11},q^{-14},q^{-17}}{q^{15}}}
		\SSeries{4}{-11}{15}
		\\&
		+\frac{\jacprod{q^{-22},q^{-21},q^{-17},q^{-16},q^{3},q^{4}}{q^{15}}}
			{\jacprod{q^{-6},q^{-3},q^{3},q^{6},q^{-8},q^{-11},q^{-14},q^{-17},q^{-20}}{q^{15}}}
		\SSeries{7}{1}{15}
		\\&
		-q^{7}\frac{\jacprod{q^{-8},q^{-7},q^{-3},q^{-2},q^{17},q^{18}}{q^{15}}}
			{\jacprod{q^{8},q^{11},q^{14},q^{17},q^{20},q^{6},q^{3},q^{-3},q^{-6}}{q^{15}}}
		\SSeries{7}{10}{15}
		\\&
		+\frac{\jacprod{q^{-28},q^{-27},q^{-23},q^{-22},q^{-3},q^{-2}}{q^{15}}}
			{\jacprod{q^{-12},q^{-9},q^{-6},q^{-3},q^{-14},q^{-17},q^{-20},q^{-23},q^{-26}}{q^{15}}}
		\SSeries{13}{25}{15}
		\\&
		-q^{13}\frac{\jacprod{q^{-2},q^{-1},q^{3},q^{4},q^{23},q^{24}}{q^{15}}}
			{\jacprod{q^{14},q^{17},q^{20},q^{23},q^{26},q^{12},q^{9},q^{6},q^{3}}{q^{15}}}
		\SSeries{13}{34}{15}
.
\end{align*}
Simplifying the products yields
\begin{align}
	\label{5DissectionEq5a}
	0 
	=&
		-q^{-13}\frac{1}{\jacprod{q^{3},q^{5},q^{6}}{q^{15}}}\SSeries{1}{-14}{15}
		-q^{-11}\frac{1}{\jacprod{q^{3},q^{5},q^{6}}{q^{15}}}\SSeries{4}{-11}{15}
		\nonumber\\&
		-q^{-6}\frac{1}{\jacprod{q^{3},q^{5},q^{6}}{q^{15}}}\SSeries{7}{1}{15}
		-q^{-2}\frac{\jacprod{q^{2},q^{7}}{q^{15}}}
			{\jacprod{q,q^{4},q^{5},q^{6},q^{6}}{q^{15}}}\SSeries{7}{10}{15}
		\nonumber\\&
		-q^{11}\frac{\jacprod{q^{2},q^{7}}{q^{15}}}
		{\jacprod{q,q^{4},q^{5},q^{6},q^{6}}{q^{15}}}\SSeries{13}{25}{15}
		-q^{19}\frac{1}{\jacprod{q^{3},q^{5},q^{6}}{q^{15}}}\SSeries{13}{34}{15}
.
\end{align}
We see multiplying (\ref{5DissectionEq5a})
by $q^{13}\jacprod{q^3,q^5,q^6}{q^{15}}$
implies (\ref{5DissectionEq5}).

In Lemma \ref{ChanLemma4} we use $q\mapsto q^{15}$, $a_1=q^{-13}$, 
$a_2 = q^{-11}$, $a_3 = q^{-8}$, $a_4 = q^{-6}$, $a_5 = q^{-3}$, $a_6 = q^{14}$,
$b_1 = q$, $b_2 = q^{4}$, $b_3 = q^{7}$, $b_4 = q^{10}$, $b_5 = q^{13}$
and note four terms on the right in 
Lemma \ref{ChanLemma4} are immediately zero, yielding
\begin{align*}
	&\frac{\jacprod{q^{-13},q^{-11},q^{-8},q^{-6},q^{-3},q^{14}}{q^{15}}\aqprod{q^{15}}{q^{15}}{\infty}}
		{\jacprod{q,q^{4},q^{7},q^{10},q^{13},q^{-1},q^{-4},q^{-7},q^{-10},q^{-13}}{q^{15}}}
	\\
	&=
		\frac{\jacprod{q^{-14},q^{-12},q^{-9},q^{-7},q^{-4},q^{13}}{q^{15}}}
			{\jacprod{q^{3},q^{6},q^{9},q^{12},q^{-2},q^{-5},q^{-8},q^{-11},q^{-14}}{q^{15}}}
		\SSeries{1}{-23}{15}
		\\&\quad
		-q^{4}\frac{\jacprod{q^{-9},q^{-7},q^{-4},q^{-2},q,q^{18}}{q^{15}}}
			{\jacprod{q^{5},q^{8},q^{11},q^{14},q^{17},q^{3},q^{-3},q^{-6},q^{-9}}{q^{15}}}
		\SSeries{4}{-2}{15}
		\\&\quad
		-q^{7}\frac{\jacprod{q^{-6},q^{-4},q^{-1},q,q^{4},q^{21}}{q^{15}}}
			{\jacprod{q^{8},q^{11},q^{14},q^{17},q^{20},q^{6},q^{3},q^{-3},q^{-6}}{q^{15}}}
		\SSeries{7}{10}{15}
		\\&\quad
		+\frac{\jacprod{q^{-23},q^{-21},q^{-18},q^{-16},q^{-13},q^{4}}{q^{15}}}
			{\jacprod{q^{-9},q^{-6},q^{-3},q^{3},q^{-11},q^{-14},q^{-17},q^{-20},q^{-23}}{q^{15}}}
		\SSeries{10}{13}{15}
		\\&\quad
		-q^{10}\frac{\jacprod{q^{-3},q^{-1},q^{2},q^{4},q^{7},q^{24}}{q^{15}}}
			{\jacprod{q^{11},q^{14},q^{17},q^{20},q^{23},q^{9},q^{6},q^{3},q^{-3}}{q^{15}}}
		\SSeries{10}{22}{15}
		\\&\quad
		+\frac{\jacprod{q^{-26},q^{-24},q^{-21},q^{-19},q^{-16},q}{q^{15}}}
			{\jacprod{q^{-12},q^{-9},q^{-6},q^{-3},q^{-14},q^{-17},q^{-20},q^{-23},q^{-26}}{q^{15}}}
		\SSeries{13}{25}{15}
.
\end{align*}
Simplifying the products yields
\begin{align}
	\label{5DissectionEq6a}
	&q^{-6}\frac{\jacprod{q^{3},q^{6}}{q^{15}}\aqprod{q^{15}}{q^{15}}{\infty}^2}
		{\jacprod{q,q^{2},q^{4},q^{5},q^{5},q^{7}}{q^{15}}}
	\nonumber\\
	&=
		q^{-6}\frac{1}{\jacprod{q^{3},q^{5},q^{6}}{q^{15}}}\SSeries{1}{-23}{15}
		+
		q^{-1}\frac{1}{\jacprod{q^{3},q^{5},q^{6}}{q^{15}}}\SSeries{4}{-2}{15}
		\nonumber\\&\quad
		-q^6\frac{\jacprod{q,q^{4}}{q^{15}}}
			{\jacprod{q^{2},q^{3},q^{3},q^{5},q^{7}}{q^{15}}}\SSeries{7}{10}{15}
		+
		q^9\frac{1}{\jacprod{q^{3},q^{5},q^{6}}{q^{15}}}\SSeries{10}{13}{15}
		\nonumber\\&\quad
		+q^{15}\frac{1}{\jacprod{q^{3},q^{5},q^{6}}{q^{15}}}\SSeries{10}{22}{15}
		-
		q^{19}\frac{\jacprod{q,q^{4}}{q^{15}}}
			{\jacprod{q^{2},q^{3},q^{3},q^{5},q^{7}}{q^{15}}}\SSeries{13}{25}{15}
.
\end{align}
\sloppy
We see multiplying (\ref{5DissectionEq6a}) 
by $q^6\jacprod{q^3,q^5,q^6}{q^{15}}$
implies (\ref{5DissectionEq6}), again noting that
$\frac{\jacprod{q^{3},q^3,q^6,q^{6}}{q^{15}}\aqprod{q^{15}}{q^{15}}{\infty}^2}
		{\jacprod{q^{1},q^{2},q^{4},q^{5},q^{7}}{q^{15}}}
=\frac{\aqprod{q^3}{q^3}{\infty}^3}{\aqprod{q}{q}{\infty}}
$.

\fussy
\end{proof}

Similar to the proof of Theorem \ref{theorem2}, we begin with expressing
$\ST(\zeta_5,q)$ in terms of $U_5(b)$.
With $\zeta_5$ a primitive fifth root of unity, we have
\begin{align}
	\label{theorem3Eq1}
	\ST(\zeta_5,q)
	&=
	\frac{1}{\aqprod{q}{q}{\infty}}
	\sum_{n=-\infty}^\infty
		\frac{q^{6n^2+4n+1}(1-q^{6n+2})(1-q^{3n+1})(1-\zeta_5^2q^{3n+1})(1-\zeta_5^3q^{3n+1}) }
		{1-q^{15n+5}}
	\nonumber\\
	&=
	\frac{1}{\aqprod{q}{q}{\infty}}
	\left(
		qU_5(4)
		+(\zeta_5+\zeta_5^4)q^2U_5(7)
		-(1+\zeta_5+\zeta_5^4)q^3U_5(10)
		-(1+\zeta_5+\zeta_5^4)q^4U_5(13)
		\right.\nonumber\\&\qquad\qquad\qquad \left.
		+(\zeta_5+\zeta_5^4)q^5U_5(16)
		+q^6U_5(19)
		\right)
.
\end{align}

We find that
\begin{align*}
	U_5(b) 
	&= 
		\sum_{n=-\infty}^\infty\frac{q^{6n^2+bn}}{1-q^{15n+5}}
	\\
	&=
		\sum_{k=0}^4 q^{6k^2+bk} \sum_{n=-\infty}^\infty 
		\frac{q^{150n(n+1)}q^{60nk+5bn-150n}}{1-q^{15k+5}q^{75n}}
	\\
	&=
		\sum_{k=0}^4 q^{6k^2+bk} \sum_{n=-\infty}^\infty 
		\SSeries{15k+5}{60k+5b-150}{75}
.
\end{align*}
Thus
\begin{align*}
	&qU_5(4)
	-q^3U_5(10)
	-q^4U_5(13)
	+q^6U_5(19)
	\\
	&=
		q\Sigma(5,-130,75)
		+q^{11}\Sigma(20,-70,75)
		+q^{33}\Sigma(35,-10,75)
		+q^{67}\Sigma(50,50,75)
		+q^{113}\Sigma(65,110,75)
		\\&\quad
 	 	-q^3\Sigma(5,-100,75)
		-q^{19}\Sigma(20,-40,75)
		-q^{47}\Sigma(35,20,75)
		-q^{87}\Sigma(50,80,75)
		-q^{139}\Sigma(65,140,75)
		\\&\quad	
		-q^4\Sigma(5,-85,75)
		-q^{23}\Sigma(20,-25,75)
		-q^{54}\Sigma(35,35,75)
		-q^{97}\Sigma(50,95,75)
		-q^{152}\Sigma(65,155,75)
		\\&\quad
		+q^6\Sigma(5,-55,75)
		+q^{31}\Sigma(20,5,75)
		+q^{68}\Sigma(35,65,75)
		+q^{117}\Sigma(50,125,75)
		+q^{178}\Sigma(65,185,75)
,
\end{align*}
and
\begin{align*}
	&q^2U_5(7) - q^3U_5(10) - q^4U_5(13) + q^5U_5(16)
	\\
	&=
		q^2\Sigma(5,-115,75)
		+q^{15}\Sigma(20,-55,75)
		+q^{40}\Sigma(35,5,75)
		+q^{77}\Sigma(50,65,75)
		+q^{126}\Sigma(65,125,75)
		\\&\quad
		-q^3\Sigma(5,-100,75)
		-q^{19}\Sigma(20,-40,75)
		-q^{47}\Sigma(35,20,75)
		-q^{87}\Sigma(50,80,75)
		-q^{139}\Sigma(65,140,75)
		\\&\quad	
		-q^4\Sigma(5,-85,75)
		-q^{23}\Sigma(20,-25,75)
		-q^{54}\Sigma(35,35,75)
		-q^{97}\Sigma(50,95,75)
		-q^{152}\Sigma(65,155,75)
		\\&\quad
		+q^5\Sigma(5,-70,75)
		+q^{27}\Sigma(20,-10,75)
		+q^{61}\Sigma(35,50,75)
		+q^{107}\Sigma(50,110,75)
		+q^{165}\Sigma(65,170,75)
.
\end{align*}

Next we reorder the $\SSeries{z}{w}{q}$
terms and apply Propositions \ref{Prop1Zeta5} and \ref{Prop2Zeta5}
with $q\mapsto q^5$
to get that
\begin{align}
	&qU_5(4)
	-q^3U_5(10)
	-q^4U_5(13)
	+q^6U_5(19)
	\nonumber\\
	&=
		q^{67}\Sigma(50,50,75)
		+q^{117}\Sigma(50,125,75)
		 -q^3\Sigma(5,-100,75)-q^{23}\Sigma(20,-25,75)
		\nonumber\\&\quad
		+q^{33}\Parans{ \Sigma(35,-10,75) +q^{35}\Sigma(35,65,75)
			+q^{80}\Sigma(65,110,75)+q^{145}\Sigma(65,185,75)
		}
		\nonumber\\&\quad
		+q\Parans{\Sigma(5,-130,75)+q^5\Sigma(5,-55,75)
			+q^{10}\Sigma(20,-70,75)	+q^{30}\Sigma(20,5,75)
		}
		\nonumber\\&\quad
		-q^4\Parans{\Sigma(5,-85,75)+q^{15}\Sigma(20,-40,75)
			+q^{50}\Sigma(35,35,75)+q^{135}\Sigma(65,140,75)
		}
		\nonumber\\&\quad
		-q^{47}\Parans{\Sigma(35,20,75)+q^{40}\Sigma(50,80,75)
			+q^{50}\Sigma(50,95,75)+q^{105}\Sigma(65,155,75)
		}
	\nonumber\\
	&=
		\Parans{q^{66}\Sigma(50,50,75)+q^{116}\Sigma(50,125,75)}
		\Parans{-\frac{\jacprod{q^{10}}{q^{15}}}{\jacprod{q^{5}}{q^{15}}}   
			+q+q^2\frac{\jacprod{q^{5}}{q^{15}}}{\jacprod{q^{10}}{q^{15}}}}
		\nonumber\\&\quad
		 -\Parans{q^2\Sigma(5,-100,75)+q^{22}\Sigma(20,-25,75)}
		\Parans{-\frac{\jacprod{q^{10}}{q^{15}}}{\jacprod{q^{5}}{q^{15}}}   
			+q+q^2\frac{\jacprod{q^{5}}{q^{15}}}{\jacprod{q^{10}}{q^{15}}}}
		\nonumber\\&\quad
		-q^3\frac{\aqprod{q^{15}}{q^{15}}{\infty}^3}
			{\aqprod{q^{25}}{q^{25}}{\infty}\jacprod{q^{10}}{q^{25}}^2}
		+q\frac{\aqprod{q^{15}}{q^{15}}{\infty}^3}
			{\aqprod{q^{25}}{q^{25}}{\infty}\jacprod{q^{5}}{q^{25}}^2}
		-q^2\frac{\aqprod{q^{15}}{q^{15}}{\infty}^3}
			{\aqprod{q^{5}}{q^{5}}{\infty}}
	\nonumber\\
	&=
		\Parans{q^{66}\Sigma(50,50,75)+q^{116}\Sigma(50,125,75)}
		\Parans{-\frac{\jacprod{q^{10}}{q^{15}}}{\jacprod{q^{5}}{q^{15}}}   
			+q+q^2\frac{\jacprod{q^{5}}{q^{15}}}{\jacprod{q^{10}}{q^{15}}}}
		\nonumber\\&\quad
		 -\Parans{q^2\Sigma(5,-100,75)+q^{22}\Sigma(20,-25,75)}
		\Parans{-\frac{\jacprod{q^{10}}{q^{15}}}{\jacprod{q^{5}}{q^{15}}}   
			+q+q^2\frac{\jacprod{q^{5}}{q^{15}}}{\jacprod{q^{10}}{q^{15}}}}
		\nonumber\\&\quad
		+q\frac{\aqprod{q^{15}}{q^{15}}{\infty}^3}
			{\aqprod{q^{5}}{q^{5}}{\infty}}	
		\Parans{\frac{\jacprod{q^{10}}{q^{15}}}{\jacprod{q^{5}}{q^{15}}}   
			-q-q^2\frac{\jacprod{q^{5}}{q^{15}}}{\jacprod{q^{10}}{q^{15}}}}
.
\end{align}
However, by Euler's pentagonal number theorem and the Jacobi triple product 
identity, we have that
\begin{align*}
	\aqprod{q}{q}{\infty}
	&=
	\aqprod{q^{25}}{q^{25}}{\infty}
	\Parans{\frac{\jacprod{q^{10}}{q^{25}}}{\jacprod{q^{5}}{q^{25}}}   
			-q-q^2\frac{\jacprod{q^{5}}{q^{25}}}{\jacprod{q^{10}}{q^{25}}}}
.
\end{align*}
Thus,
\begin{align}
	\label{theorem3Eq2}
	&qU_5(4)
	-q^3U_5(10)
	-q^4U_5(13)
	+q^6U_5(19)
	\nonumber\\
	&=
		\frac{\aqprod{q}{q}{\infty}}{\aqprod{q^{25}}{q^{25}}{\infty}}
		\Parans{q^2\Sigma(5,-100,75)+q^{22}\Sigma(20,-25,75)
			-q^{66}\Sigma(50,50,75)-q^{116}\Sigma(50,125,75)}
		\nonumber\\&\quad			
		+q\frac{\aqprod{q}{q}{\infty}\aqprod{q^{15}}{q^{15}}{\infty}^3}
			{\aqprod{q^{5}}{q^{5}}{\infty}\aqprod{q^{25}}{q^{25}}{\infty}}	
.		
\end{align}

Similarly we reorder the $\SSeries{z}{w}{q}$
terms and apply Propositions \ref{Prop2Zeta5} and
\ref{Prop3Zeta5} with $q\mapsto q^5$
to get that
\begin{align}
	&q^2U_5(7) - q^3U_5(10) - q^4U_5(13) + q^5U_5(16)
	\nonumber\\
	&=
		q^{61}\Sigma(35,50,75)		+q^{126}\Sigma(65,125,75)
		-q^3\Sigma(5,-100,75)-q^{23}\Sigma(20,-25,75)
		\nonumber\\&\quad
		+q^2\Parans{\Sigma(5,-115,75) + q^{25}\Sigma(20,-10,75) 
			+ q^{75}\Sigma(50,65,75) + q^{105}\Sigma(50,110,75)
		}
		\nonumber\\&\quad
		+q^5\Parans{\Sigma(5,-70,75)+q^{10}\Sigma(20,-55,75)
			+q^{35}\Sigma(35,5,75)+q^{160}\Sigma(65,170,75)
		}
		\nonumber\\&\quad
		-q^4\Parans{\Sigma(5,-85,75)+q^{15}\Sigma(20,-40,75)
			+q^{50}\Sigma(35,35,75)+q^{135}\Sigma(65,140,75)
		}
		\nonumber\\&\quad
		-q^{47}\Parans{\Sigma(35,20,75)+q^{40}\Sigma(50,80,75)
			+q^{50}\Sigma(50,95,75)+q^{105}\Sigma(65,155,75)
		}
	\nonumber\\	
	&=
		\left( q^{60}\Sigma(35,50,75) + q^{125}\Sigma(65,125,75) \right)
		\Parans{-\frac{\jacprod{q^{10}}{q^{15}}}{\jacprod{q^{5}}{q^{15}}}   
			+q+q^2\frac{\jacprod{q^{5}}{q^{15}}}{\jacprod{q^{10}}{q^{15}}}}
		\nonumber\\&\quad
		 -\Parans{q^2\Sigma(5,-100,75)+q^{22}\Sigma(20,-25,75)}
		\Parans{-\frac{\jacprod{q^{10}}{q^{15}}}{\jacprod{q^{5}}{q^{15}}}   
			+q+q^2\frac{\jacprod{q^{5}}{q^{15}}}{\jacprod{q^{10}}{q^{15}}}}
		\nonumber\\&\quad
		+q^2\frac{\aqprod{q^{15}}{q^{15}}{\infty}^3}
			{\aqprod{q^{5}}{q^{5}}{\infty}}
		-q^2\frac{\aqprod{q^{15}}{q^{15}}{\infty}^3}
			{\aqprod{q^{5}}{q^{5}}{\infty}}
	\nonumber\\
	&=
		\frac{\aqprod{q}{q}{\infty}}{\aqprod{q^{25}}{q^{25}}{\infty}}
		\Parans{q^2\Sigma(5,-100,75)+q^{22}\Sigma(20,-25,75)
			-q^{60}\Sigma(35,50,75) - q^{125}\Sigma(65,125,75)
		}
\label{theorem3Eq3}
.		
\end{align}
Theorem \ref{theorem3} follows by equations (\ref{theorem3Eq1}), 
(\ref{theorem3Eq2}), and (\ref{theorem3Eq3}).

\section{Proof of Theorem \ref{crankTheorem}}
\begin{proof}
Using Theorem 2.1 of \cite{AndrewsBook} we find that
\begin{align}
	\ST(z,q)
	&=
	\sum_{n=1}^\infty \frac{q^n}{\aqprod{zq^n}{q}{\infty}}
		\sum_{k=0}^\infty \frac{\aqprod{zq}{q}{k}z^{-k}q^{nk}}{\aqprod{q}{q}{k}}
	\nonumber\\	
	&=
	\sum_{n=1}^\infty\sum_{k=0}^\infty \frac{z^{-k} q^{n+kn}}
		 {\aqprod{zq^{n+k}}{q}{\infty}\aqprod{q}{q}{k}}
	\nonumber\\
	&=
		\sum_{n=1}^\infty \frac{q^{n}}
		 {\aqprod{zq^{n}}{q}{\infty}}
		+	\sum_{n=1}^\infty\sum_{k=1}^\infty \frac{z^{-k} q^{n+kn}}
		 {\aqprod{zq^{n+k}}{q}{\infty}\aqprod{q}{q}{k}}
	\nonumber\\
	&=
		\sum_{n=1}^\infty \frac{q^{n}}
		 {\aqprod{zq^{n}}{q}{\infty}}
		+
		\sum_{n=1}^\infty\sum_{k=1}^\infty 
		\frac{q^{n}}{\aqprod{q^{n}}{q}{k}\aqprod{zq^{n+k}}{q}{\infty}}
		\cdot
		\frac{z^{-k} q^{kn}\aqprod{q}{q}{n+k-1}}
		 {\aqprod{q}{q}{n-1}\aqprod{q}{q}{k}}
\label{CrankIdentity}
.
\end{align}
The first term in (\ref{CrankIdentity}) is the generating function for 
non-empty partitions with the 
power of $q$ giving the number being partitioned and the power of $z$ giving one
less than the number of parts of the partition. This is the generating function
for partition pairs from $\ST$ where $\pi_2$ is empty, the power of $q$ gives
the number being partitioned and the power of $z$ gives the paircrank. 
	
In the second term in (\ref{CrankIdentity}), we interpret the summands 
as follows.
We have that $\frac{q^{n}}{\aqprod{q^{n}}{q}{k}\aqprod{zq^{n+k}}{q}{\infty}}$ 
is the generating function for
non-empty partitions
$\pi_1$
where $n$ is the smallest part,
the power of $q$ counts the number being partitioned, and
the power of $z$ counts the number 
of parts that are at least $n+k$.
We have that
$\frac{z^{-k} q^{kn}\aqprod{q}{q}{n+k-1}}{\aqprod{q}{q}{n-1}\aqprod{q}{q}{k}}$ 
is the generating function for partitions $\pi_2$
with exactly $k$ parts with each part
between $n$ and $2n-1$,
where the power of $q$ counts the number being partitioned and the power of $z$
counts the negative of the number of parts.
Thus the second term in (\ref{CrankIdentity}) is the
generating function for partition pairs from $\ST$, with both $\pi_1$ and 
$\pi_2$ non-empty, with the power of $q$ giving
the number being partitioned and the power of $z$ giving the paircrank.

This proves that $C(m,n)$ is the number of partition pairs from $\ST$
of $n$ with paircrank $m$. These rearrangements and interpretations
are essentially what was done in \cite{GarvanJennings} as an intermediate
step in getting an spt-crank defined on marked overpartitions. Also this is 
quite similar to the steps in \cite{AndrewsGarvan} where Andrews and Garvan gave the
ordinary partition crank after the vector partition crank.
\end{proof}

\section*{Acknowledgments}
The author would like to thank Frank Garvan for suggesting this problem
and for his help and encouragement.

\bibliographystyle{abbrv}
\bibliography{pairsWithConditionsRef}

\end{document}